\documentclass{amsart}

\usepackage{amsthm, amssymb, amsmath, stmaryrd, mathrsfs}

\input xy
\xyoption{all}

\newcommand{\quash}[1]{}  %%Anything in \quash is ignored
\newcommand{\leftexp}[2]{{\vphantom{#2}}^{#1}{#2}}

\newcommand{\const}[1]{\overline{\QQ}_{\ell,#1}}
\newcommand{\twtimes}[1]{\stackrel{#1}{\times}}

\newcommand{\homo}[2]{\mathbf{H}_{#1}({#2})}   % sheaf
\newcommand{\homog}[2]{\textup{H}_{#1}({#2})}  % plain group
\newcommand{\coho}[2]{\mathbf{H}^{#1}({#2})}    % sheaf
\newcommand{\cohog}[2]{\textup{H}^{#1}({#2})}     % plain group
\newcommand{\hBM}[2]{\textup{H}^{\textup{BM}}_{#1}({#2})}  % Borel-Moore

\DeclareMathOperator{\GL}{GL}
\DeclareMathOperator{\SL}{SL}

\DeclareMathOperator{\Hom}{Hom}

\DeclareMathOperator{\End}{End}
\DeclareMathOperator{\Aut}{Aut}
\DeclareMathOperator{\Corr}{Corr}
\DeclareMathOperator{\Sym}{Sym}
\DeclareMathOperator{\Gr}{Gr}

\DeclareMathOperator{\Spec}{Spec}

\DeclareMathOperator{\id}{id}

\DeclareMathOperator{\ev}{ev}

\DeclareMathOperator{\Pic}{Pic}

\DeclareMathOperator{\Lie}{Lie}

\DeclareMathOperator{\triv}{triv}
\DeclareMathOperator{\codim}{codim}

\def\bH{\mathbf{H}}
\def\bI{\mathbf{I}}

\def\bR{\mathbf{R}}

\def\BB{\mathbb{B}}
\def\CC{\mathbb{C}}
\def\DD{\mathbb{D}}

\def\FF{\mathbb{F}}
\def\GG{\mathbb{G}}

\def\PP{\mathbb{P}}
\def\QQ{\mathbb{Q}}

\def\ZZ{\mathbb{Z}}

\def\frg{\mathfrak{g}}
\def\frt{\mathfrak{t}}

\def\frb{\mathfrak{b}}

\def\frc{\mathfrak{c}}
\def\frm{\mathfrak{m}}

\def\frX{\mathfrak{X}}
\def\frY{\mathfrak{Y}}

\def\calA{\mathcal{A}}
\def\calB{\mathcal{B}}

\def\calE{\mathcal{E}}
\def\calF{\mathcal{F}}
\def\calG{\mathcal{G}}
\def\calH{\mathcal{H}}

\def\calK{\mathcal{K}}
\def\calL{\mathcal{L}}
\def\calM{\mathcal{M}}

\def\calO{\mathcal{O}}
\def\calP{\mathcal{P}}
\def\calQ{\mathcal{Q}}

\def\tilg{\widetilde{\mathfrak{g}}}

\def\tils{\widetilde{s}}
\def\tilw{\widetilde{w}}
\def\tilW{\widetilde{W}}
\def\tila{\widetilde{a}}

\def\tilx{\widetilde{x}}

\def\tilf{\widetilde{f}}

\def\tcP{\widetilde{\calP}}
\def\tcA{\widetilde{\calA}}
\def\tcK{\widetilde{\calK}}

\def\hatO{\widehat{\calO}}

\def\calHn{\calH^{\natural}}

\def\xch{\mathbb{X}^*}
\def\xcoch{\mathbb{X}_*}

\def\Ql{\overline{\QQ}_\ell}

\def\htimes{\widehat{\times}}

\def\isom{\stackrel{\sim}{\to}}

\def\pH{\leftexp{p}{\mathbf{H}}}
\def\ptau{\leftexp{p}{\tau}}

\def\act{\textup{act}}
\def\rs{\textup{rs}}
\def\reg{\textup{reg}}
\def\red{\textup{red}}

\def\Ad{\textup{Ad}}

\def\coh{\textup{coh}}

\def\St{\textup{St}}

\def\parab{\textup{par}}

\def\bch{\textup{b.c.}}
\def\adj{\textup{ad.}}
\def\stsh{*!\to!*}

\def\XR{X_R}
\def\XS{S\times X}

\def\Poin{\textup{Poin}}

\def\punc{\textup{punc}}

\def\Bun{\textup{Bun}}
\def\Hit{\textup{Hit}}
\def\Grass{\mathcal{G}r}
\def\tilGr{\widetilde{\Grass}}
\def\Flag{\mathcal{F}\ell}

\def\tilGr{\widetilde{\Grass}}

\def\Bunpar{\Bun^{\parab}}

\def\stPic{\calP\textup{ic}}

\def\ani{\textup{ani}}
\def\Ah{\calA^{\heartsuit}}
\def\Aa{\calA^{\ani}}
\def\AHit{\calA^{\Hit}}
\def\Ag{\calA^{\diamondsuit}}

\def\tcArs{\tcA^{\rs}}

\def\Mpar{\mathcal{M}^{\parab}}
\def\Mparrs{\mathcal{M}^{\parab,\rs}}
\def\Mparreg{\mathcal{M}^{\parab,\reg}}
\def\MHit{\mathcal{M}^{\Hit}}
\def\MHitreg{\mathcal{M}^{\Hit,\reg}}

\def\Hecke{\mathcal{H}\textup{ecke}}
\def\Heckep{\Hecke^{\parab}}
\def\Heckeprs{\Hecke^{\parab,\rs}}

\def\fpar{f^{\parab}}

\def\fQl{f^{\parab}_*\Ql}

\def\tfQl{\widetilde{f}_*\Ql}

\def\Wa{W_{\textup{aff}}}

\def\disk{\mathfrak{D}}
\def\pdisk{\disk^{\times}}

\def\dual{^\vee}
\def\Gd{G\dual}

\swapnumbers 
\theoremstyle{plain}
\newtheorem{theorem}[subsubsection]{Theorem}
\newtheorem{lemma}[subsubsection]{Lemma}
\newtheorem{cor}[subsubsection]{Corollary}
\newtheorem{prop}[subsubsection]{Proposition}

\newtheorem*{thm}{Theorem}
\newtheorem*{pp}{Proposition}

\theoremstyle{definition}
\newtheorem{defn}[subsubsection]{Definition}
\newtheorem{cons}[subsubsection]{Construction}
\newtheorem{remark}[subsubsection]{Remark}
\newtheorem{exam}[subsubsection]{Example}

\numberwithin{equation}{section}

\title[Towards a Global Springer Theory I]{Towards a Global Springer Theory I:\\ the affine Weyl group action}
\author{Zhiwei Yun}
\address{Department of Mathematics, Princeton University, Princeton, NJ 08544, USA}
\email{zyun@math.princeton.edu}
\date{December 2007; revised April 2009}
\subjclass[2000]{Primary 14H60, 20F55; Secondary 14F20, 20G35}

\begin{document}

\begin{abstract}
We propose a generalization of Springer representations to the context of groups over a global function field. The global counterpart of the Grothendieck simultaneous resolution is the parabolic Hitchin fibration. We construct an action of the affine Weyl group on the direct image complex of the parabolic Hitchin fibration. In particular, we get representations of the affine Weyl group on the cohomology of parabolic Hitchin fibers, providing the first step towards a global Springer theory.
\end{abstract}

\maketitle
\tableofcontents

\section{Introduction}

This is the first of a series of papers in which we propose a global analogue of the Springer theory. Here ``global'' refers to the function field of an algebraic curve.

\subsection{Springer theories: Classical, Local and Global}\label{ss:hist}

In this subsection, we will put our new theory in the appropriate historical context by briefly reviewing the classical and local versions of the Springer theories.

\subsubsection{The classical Springer theory}\label{sss:cl}

Let us start with a brief review of the classical Springer theory (\cite{Sp},\cite{L81},\cite{BM},\cite{CG}). The classical Springer theory originated from Springer's study of Green functions for finite groups of Lie type (\cite{Sp}). Let us state only the Lie algebra version of his result. Let $G$ be a reductive group over an algebraically closed field $k$ with Lie algebra $\frg$, and let $\calB$ be the flag variety classifying Borel subgroups of $G$. Let $\tilg$ be the scheme classifying pairs $(\gamma,B)$ where $\gamma\in\frg, B\in\calB$ such that $\gamma\in\Lie B$. Forgetting the choice of $B$ we get the so-called {\em Grothendieck simultaneous resolution}:
\begin{equation}\label{eq:intropi}
\pi:\tilg\to\frg.
\end{equation}
For an element $\gamma\in\frg$, the fiber $\calB_{\gamma}=\pi^{-1}(\gamma)$ is called the {\em Springer fiber} of $\gamma$. These are closed subschemes of the flag variety $\calB$. If $\gamma$ is regular semisimple, $\calB_\gamma$ is simply a $W$-torsor, where $W$ is the Weyl group of $G$. In general, $\calB_\gamma$ are of higher dimensions and have complicated singularities. Springer constructed representations of the Weyl group $W$ of $G$ on the top-dimensional cohomology of $\calB_\gamma$. Later, with the invention of Intersection Homology and perverse sheaves, a sheaf-theoretic version of Springer representations was given by Lusztig (\cite{L81}). He showed that there is an action of $W$ on the shifted perverse sheaf $\pi_*\Ql$ on the affine space $\frg$, hence incorporating the $W$-actions on the cohomology of various Springer fibers into a family. This approach was further developed by Borho-MacPherson (\cite{BM}), and they showed that all irreducible representations of $W$ arise as Springer representations. There are other constructions of Springer representations by Kazhdan-Lusztig (\cite{KL}) using the Coxeter presentation of $W$, and by Chriss-Ginzburg (\cite{CG}) using Steinberg correspondences.

One mysterious feature of Springer representations is that the $W$-action on the cohomology of $\calB_{\gamma}$ does {\em not} come from an {\em algebraic} action of $W$ on the variety $\calB_\gamma$, which makes it difficult to calculate. However, it is possible to deduce information about individual Springer representations from the knowledge of the (shifted) perverse sheaf $\pi_*\Ql$. The key fact here is that $\pi_*\Ql$ is the {\em middle extension} from a local system on the regular semisimple locus of $\frg$. Hence, in some sense, ``good'' (i.e., boring) Springer fibers control ``bad'' (i.e., interesting) ones.

% local

\subsubsection{The local Springer theory}\label{sss:histloc}

Since the classical Springer theory is related to the representation theory of $G(\FF_q)$, we can think of it as a theory ``over $\Spec\FF_q$''. It is then natural to ask whether there are corresponding theories over a local field such as $k((t))$, or over a global field such as the function field of an algebraic curve over a field $k$. Here $k$ can be any field.

A local theory (for the local function field $F=k((t))$) already exists, by work of Lusztig (\cite{L96}). In the local theory, the loop group $G((t))$ (the Weil restriction of $G$ from $k((t))$ to $k$) replaces the group $G$, the affine Weyl group $\tilW=\xcoch(T)\rtimes W$ replaces the finite Weyl group $W$, and {\em affine Springer fibers} (cf. \cite{KL88}) replace Springer fibers. For an element $\gamma\in\frg\otimes k((t))$, the affine Springer fiber $M_{\gamma}$ is the closed sub-ind-schemes of the {\em affine flag variety} $\Flag_G$ parametrizing Iwahori subgroups $\bI\subset G((t))$ such that $\gamma\in\Lie\bI$. These $M_{\gamma}$ are highly non-reduced ind-schemes, typically with infinitely many irreducible components. In \cite{L96}, Lusztig constructed actions of the affine Weyl group $\tilW$ on the homology of affine Springer fibers, using the Coxeter presentation of $\tilW$. However, unlike the classical Springer theory, the local theory does not yet have a satisfactory sheaf-theoretic approach which allows to organize affine Springer fibers into geometrically manageable families. An essential difficulty is that the parameter $\gamma$ of affine Springer fibers naturally lives in an infinite-dimensional subspace of $\frg\otimes k((t))$, on which perverse sheaves, middle extensions, etc are hard to make sense of.

% global

\subsubsection{The global Springer theory}\label{sss:global}

The goal of this series of papers is to give a sheaf-theoretic construction of a {\em global} Springer theory, i.e., we will start with a complete smooth connected curve $X$ over $k$ and a reductive group scheme $\underline{G}$ over $X$. We will mostly be working with the constant group scheme $G\times X$ (where $G$ is a connected reductive group over $k$); however, our results easily extend to the case of quasi-split group schemes over $X$, as we will see in \cite[Sec. 3.1]{GSIII}.  For convenience we will assume that the base field $k$ is algebraically closed; for later applications to the harmonic analysis of $p$-adic groups, we will take $k$ to be a finite field.

% Motivation

The global Springer theory, besides being an analogue of the classical and local Springer theories, is also inspired by two other sources of ideas. One is B-C.Ng\^{o}'s recent proof of the Fundamental Lemma (\cite{NgoFL}); the other is the ``mirror symmetric'' viewpoint of the geometric Langlands program, proposed by Kapustin-Witten (\cite{KW}).

% motivation from FL

The Fundamental Lemma is an identity of orbital integrals in $p$-adic harmonic analysis conjectured by Langlands-Shelstad, but its proof relies heavily on geometry. The geometry involved in the proof consists of a local part and a global part. The local part was studied by the pioneer work \cite{GKM} of Goresky-Kottwitz-MacPherson, where they interpreted the orbital integrals as the number of points on affine Springer fibers. The global part was initiated by Laumon-Ng\^o (\cite{LauN}) and finalized by Ng\^{o}. In \cite{NgoFL}, Ng\^o considers the (generalized) {\em Hitchin fibration} $f^{\Hit}:\calM^{\Hit}\to\AHit$ (for definition, see Sec. \ref{sss:par}). The number of points on the {\em Hitchin fibers} (fibers of $f^{\Hit}$) are given by global orbital integrals. Ng\^{o}'s work, especially the product formula \cite[Prop. 4.13.1]{NgoFL}, makes it clear that Hitchin fibers are the correct global analogue of affine Springer fibers in the affine Grassmannian $\Grass_G$. In this paper, we propose the {\em parabolic Hitchin fibers} as the global analogue of the affine Springer fibers $M_\gamma$ in the affine flag variety $\Flag_G$ mentioned in Sec. \ref{sss:histloc}. These are fibers of the {\em  parabolic Hitchin fibration}
\begin{equation*}
\fpar:\Mpar\to\AHit\times X.
\end{equation*}
For definition of the spaces involved, see Def. \ref{def:par1}. We consider $\fpar$ as the global analogue of the Grothendieck simultaneous resolution $\pi$ in (\ref{eq:intropi}). As a first step towards a global Springer theory, in Sec. \ref{s:Waff} we will construct an action of $\tilW=\xcoch(T)\rtimes W$ on the {\em parabolic Hitchin complex} $\fQl$ (more precisely, on its restriction to an open subset $\calA\times X\subset\AHit\times X$, see Rem. \ref{rm:largedeg}). Taking stalks of $\fQl$, we get representations of $\tilW$ on the cohomology of the parabolic Hitchin fibers $\Mpar_{a,x}$ for any $(a,x)\in\calA\times X$. In fact, we will do more. We will construct an action of the {\em graded double affine Hecke algebra} on the complex $\fQl$ in \cite[Sec. 3]{GSII}, extending the $\tilW$-action. 

% motivation from Mirror Sym

The other source of ideas that inspired the global Springer theory is the new perspective of the geometric Langlands program from mirror symmetry, as proposed by Kapustin-Witten in \cite{KW}. In this picture, Hitchin fibrations $\MHit_G$ and $\MHit_{\Gd}$ for the group $G$ and its Langlands dual $\Gd$ serve as ``classical limits'' of a mirror pair, and the geometric Langlands conjecture predicts an equivalence between the derived categories of coherent sheaves on $\MHit_G$ and $\MHit_{\Gd}$. Moreover, this conjectural equivalence is expected to transform 'tHooft operators on $D_{\coh}(\MHit_G)$ (coming from certain Hecke correspondences) into Wilson operators on $D_{\coh}(\MHit_{\Gd})$ (tensor product with certain tautological bundles). What we will prove in \cite[Sec. 4]{GSIII} is reminiscent of this conjecture, or rather its shadow on the level of cohomology. Roughly speaking, we will identify the stable parts of the parabolic Hitchin complexes for $G$ and $\Gd$, and show that the lattice part of the global Springer action (which are analogues of 'tHooft operators) transforms to certain Chern class action (which are analogues of Wilson operators) under this identification. In other words, the two lattice actions that are encoded in the graded double affine Hecke algebra action mentioned above get interchanged under Langlands duality.

% New features in construction--use corr

\subsection{New features of the global Springer theory}\label{ss:features}

The construction of the global Springer action are essential different from the construction of classical or local Springer actions. On one hand, compared to the affine Springer fibers in the local theory, parabolic Hitchin fibers have the advantage of naturally forming a {\em flat} family over a {\em finite-dimensional} base. This enables us to give a sheaf-theoretic construction. On the other hand, compared to the classical theory, the parabolic Hitchin complex $\fQl$ does not lie in a single perverse degree as the Springer sheaf $\pi_*\Ql$ does, hence the middle extension approach of Lusztig in \cite{L81} fails in the global situation. To get around this difficulty, we use cohomological correspondences to construct the $\tilW$-action on the complex $\fQl$. Construction of this flavor already existed for the classical Springer theory, using Steinberg correspondences (see \cite{CG} for a topological approach, see also Sec. \ref{ss:claction} for a brief review). In the global theory, we consider cohomological correspondences supported on the {\em Hecke correspondences} $\Heckep$ (see Sec. \ref{ss:Hecke}). A key geometric fact that makes this construction work is the codimension estimate of certain strata in $\calA$ (see Prop. \ref{p:globaldel}), proved by Ng\^o in \cite{NgoFL} on the basis of the work of Goresky-Kottwitz-MacPherson in \cite{GKM2}.

Although the parabolic Hitchin complex $\fQl$ fails to lie in a single perverse degree, it is (non-canonically) a direct sum of shifted simple perverse sheaves, each being a middle extension from its restriction to the regular semisimple locus of $\calA\times X$ (see the Support Theorem in \cite[Sec. 2]{GSIII}). In this sense, our previous remark on the classical Springer theory still holds in the global situation: the ``good'' (i.e., less interesting) parabolic Hitchin fibers control the bad (i.e., more interesting) ones. Here, ``good'' parabolic Hitchin fibers are disjoint unions of torsors under abelian varieties, as opposed to discrete set of points in the classical or local theory.

% New feature--non-semisimple

The global Springer theory has a few new features that are not shared by the classical theory. Some of these features should exist in the local theory but are technically more difficult to obtain. The first interesting feature is the non-semisimplicity of the affine Weyl group action. More precisely, the action of the lattice part $\xcoch(T)$ of $\tilW$ on $\cohog{*}{\Mpar_{a,x}}$ are not semisimple in general. In \cite[Sec. 5]{GSIII}, we will work out an example in the case of $G=\SL(2)$ to illustrate this phenomenon.

% New feature--more symmetry

The second new feature is that global Springer theory carries richer and more interesting symmetry than the classical and local ones. There are at least three pieces of symmetry acting on the parabolic Hitchin fibration $\fpar:\Mpar\to\calA\times X$: the first is the affine Weyl group action on $\fQl$ to be constructed in this paper; the second is the {\em cup product} of Chern classes of certain line bundles on $\Mpar$; the third is the action of a Picard stack $\calP$ on $\Mpar$ (see Sec. \ref{ss:sym}). This Picard stack action is analogous to the action of the centralizer $G_{\gamma}$ of the element $\gamma$ on $\calB_\gamma$ in the classical situation. A large part of the second paper \cite{GSII} of the series will be devoted to the study of the interplay among these three pieces of symmetry: we will show in \cite[Sec. 3]{GSII} that the first and the second pieces of symmetry together give an action of the graded double affine Hecke algebra on $\fQl$; we will also study the relation between the third piece of symmetry (the cap product) and the first two in \cite[Sec. 5]{GSII}.

% New feature--endoscopic

The third new feature is that notions such as endoscopy and Langlands duality from the modern theory of automorphic forms naturally show up in the context of global Springer theory. Endoscopic groups come into the picture when we try to understand the direct summands of $\fQl$ which are supported on proper closed subsets of $\calA\times X$. In \cite[Sec. 3]{GSIII} we will prove the Endoscopic Decomposition Theorem, which reduces the study of these direct summands to groups smaller than $G$, the {\em endoscopic groups}. The Langlands dual group comes into the story when we try to understand the global Springer action on those direct summands of $\fQl$ which are support on the whole $\calA\times X$. We will prove in \cite[Sec. 4]{GSIII} that the lattice part of the global Springer action can be understood via certain Chern classes acting on the parabolic Hitchin complex for the Langlands dual group $\Gd$, and {\em vice versa}.

\subsection{Main Results}

Let $X$ be a projective smooth connected curve over an algebraically closed field $k$. Let $G$ be a reductive group over $k$. Fix a Borel subgroup $B$ of $G$. Fix a divisor $D$ on $X$ with $\deg(D)\geq2g_X$ ($g_X$ is the genus of $X$). The {\em parabolic Hitchin moduli stack} $\Mpar=\Mpar_{G,X,D}$ is the algebraic stack which classifies quadruples $(x,\calE,\varphi,\calE^B_x)$, where $x$ is a point on $X$, $\calE$ is a $G$-torsor over $X$, $\varphi$ is a global section of the vector bundle $\Ad(\calE)\otimes\calO_X(D)$ on $X$, and $\calE^B_x$ is a $B$-reduction of the restriction of $\calE$ at $x$ compatible with $\varphi$. For a concrete description of this stack in the case of $G=\GL(n)$, see Example \ref{ex:gl}.

Let $T$ be the quotient torus of $B$ (the universal Cartan) and $\frt$ be its Lie algebra. Let $\frc=\frt\sslash W=\frg\sslash G=\Spec k[f_1,\cdots,f_n]$ be the GIT adjoint quotient, with fundamental invariants $f_1,\cdots,f_n$ of degree $d_1,\cdots,d_n$. Let $\AHit=\AHit_{G,X,D}$ be the {\em Hitchin base}:
\begin{equation*}
\AHit=\bigoplus_{i=1}^n \cohog{0}{X,\calO_X(d_iD)}.
\end{equation*}

The morphism
\begin{eqnarray*}
\fpar:\Mpar&\to&\AHit\times X\\
(x,\calE,\varphi,\calE^B_x)&\mapsto&(f_1(\varphi),\cdots,f_n(\varphi),x)
\end{eqnarray*}
is called the {\em parabolic Hitchin fibration}. In the case of $\GL(n)$, the fibers of $\fpar$ can be described in terms of the compactified Picard stack of the spectral curves, see Example \ref{ex:gl}. The direct image complex $\fQl$ of the constant sheaf $\Ql$ under the morphism $\fpar$ is called the {\em parabolic Hitchin complex}. In the following, we will restrict the parabolic Hitchin complex to an open subset $\calA$ of $\AHit$. If $\textup{char}(k)=0$, we can $\calA$ to be the anisotropic locus $\calA^{\ani}$ of $\AHit$ (see Definition \ref{def:ani}). If $\textup{char}(k)>0$, we let $\calA\subset\calA^{\ani}$ to be the locus where the codimension estimate in Prop. \ref{p:globaldel} holds. It is expected by Ng\^o that this codimension estimate holds unconditionally in any characteristic, in which case we could always take $\calA=\Aa$. In any case, $\calA$ is nonempty if and only $G$ is semisimple. We automatically restrict all stacks over $\AHit$ to the open subset $\calA$, without changing notations.

In Sec. \ref{ss:geom}, we show that $\Mpar$ is a smooth Deligne-Mumford stack and $\fpar$ is a proper morphism. Our first main result, which justifies the phrase ``global Springer theory'' in the title, constructs an action of the affine Weyl group $\tilW=\xcoch(T)\rtimes W$ (where $\xcoch(T)$ is the cocharacter lattice of $T$) on the complex $\fQl$.

\begin{thm}[See Th. \ref{th:action}]
For each $\tilw\in\tilW$, there is a self-correspondence $\calH_{\tilw}$ of $\Mpar$ over $\calA\times X$ which is generically a graph. The fundamental class $[\calH_{\tilw}]$ of $\calH_{\tilw}$ gives an endomorphism $[\calH_{\tilw}]_\#$ of $\fQl$. The assignment $\tilw\mapsto[\calH_{\tilw}]_\#$ gives an action of $\tilW$ on $\fQl$.
\end{thm}

For the meaning of $[\calH_{\tilw}]_\#$ (i.e., the action of a cohomological correspondence on complexes of sheaves), see Sec. \ref{ss:corr}.

% A'-tfQl

There is a variant of this theorem. We can define an ``enhanced'' version of the parabolic Hitchin fibration and the parabolic Hitchin complex. In fact, the morphism $\fpar$ factors as
\begin{equation*}
\fpar:\Mpar\xrightarrow{\tilf}\tcA\xrightarrow{q}\calA\times X
\end{equation*}
where $q:\tcA\to\calA\times X$ is a branched $W$-cover, called the {\em universal cameral cover}. We call the morphism $\tilf$ the {\em  enhanced parabolic Hitchin fibration} and the complex $\tfQl$ the {\em enhanced parabolic Hitchin complex}.

\begin{pp}[see Prop. \ref{p:enhancedaction}]
There is a natural $\tilW$-equivariant structure on $\tfQl$ compatible with the action of $\tilW$ on $\tcA$ via the quotient $\tilW\twoheadrightarrow W$.
\end{pp}

All these results can be generalized to quasi-split group schemes over $X$ whose split center is trivial. We will make remarks on this in \cite[Sec. 3.1]{GSIII}.

This paper is organized as follows. In Sec. \ref{s:not}, we fix notations for all the papers in this series. In Sec. \ref{s:Mpar}, we define and study the geometry of the parabolic Hitchin fibrations. In Sec. \ref{s:Waff}, we prove the two main results above, after introducing Hecke correspondences between parabolic Hitchin moduli stacks. In App. \ref{s:corr}, we review the general formalism of cohomological correspondences, with emphasis on {\em graph-like} correspondences (see subsection \ref{ss:grlike}).

\subsection{Preview of \cite{GSII}, \cite{GSIII}}
In \cite{GSII}, we will extend the $\tilW$-action on $\fQl$ to the action of the graded {\em double affine Hecke algebra} (DAHA) on $\fQl$. We also generalize this action to ``parahoric versions'' of Hitchin stacks. We then study the interaction of the graded DAHA action and the cap product action given by the Picard stack $\calP$.

In \cite{GSIII}, we study the decomposition of the complex $\fQl$ according to characters of $\xcoch(T)$. We will prove the Endoscopic Decomposition Theorem, which links certain direct summands of $\fQl$ to the endoscopic groups of $G$. This result generalizes Ng\^o's geometric stabilization of the trace formula in \cite{NgoFL}. The second result links the stable parts of the parabolic Hitchin complexes for Langlands dual groups, and establishes a relation between the global Springer action on one hand and certain Chern class action on the other. As we mentioned in Sec. \ref{sss:global}, this result is inspired by the mirror symmetry between dual Hitchin fibrations. Finally, we present the first nontrivial example in the global Springer theory.

We will make remarks on the applications of the Endoscopic Decomposition Theorem in the Introduction of \cite{GSIII}.

\subsection*{Acknowledgment}
The author benefitted a lot from the lectures given by Ng\^{o} Bao Ch\^au on the Fundamental Lemma during the fall semesters of 2006 and 2007, as well as some subsequent discussions. The author thanks R.MacPherson and M.Goresky for their encouragement and for patiently listening to his presentations in the early stage of this work. The author also thanks G.Lusztig for help discussions and R.Bezrukavnikov, D.Gaitsgory, D.Nadler, D.Treumann and X.Zhu for their interest in the topic.

\section{Notations and conventions}\label{s:not}

In this section, we fix notations and conventions for all the papers in the series.

\subsection{General notations}\label{not:gen}
Throughout these papers, we work over a fixed algebraically closed field $k$. All fiber products of stacks without specified base are understood to be fiber products over $k$. We fix a prime $\ell$ different from $\textup{char}(k)$.

All torsors are right torsors unless otherwise stated.

Suppose $\frX$ (resp. $\frY$) is a stack with right (resp. left) action of a group scheme $A$ over $k$, then we write
\begin{equation*}
\frX\twtimes{A}\frY:=[(\frX\times\frY)/A]
\end{equation*}
for the stack quotient of $\frX\times\frY$ by the anti-diagonal right $A$-action: $a\in A(R)$ acts on $\frX(R)\times\frY(R)$ by $(x,y)\mapsto(xa,a^{-1}y)$, for any $k$-algebra $R$ and $x\in\frX(R),y\in\frY(R)$. 

For a group $A$ and a group $B$ acting on $A$ via a homomorphism $\rho:B\to\Aut(A)$, we can form the semidirect product $A\rtimes^{\rho}B$. When the action $\rho$ is clear from the context, we often write $A\rtimes B$ for simplicity.

For a group $A$ acting on an abelian group $A$, let $V^A\subset V$ be the invariants and $V\twoheadrightarrow V_A$ be the coinvariants under the $A$-action.

For an algebraic torus $T$, let $\xcoch(T)$ and $\xch(T)$ be its cocharacter and character lattices.

For a commutative group scheme (or a Picard stack which is Deligne-Mumford) $P$ over $S$ with connected fibers, let $V_\ell(P/S)=T_\ell(P/S)\otimes_{\ZZ_\ell}\Ql$ be the sheaf of $\Ql$-Tate modules of $P$ over $S$.

For a morphism of schemes $f:\frX\to\frY$, let $\stPic(\frX/\frY)$ be the Picard stack over $\frY$ classifying line bundles along the fibers of $f$. Let $\Pic(\frX/\frY)$, if exists, be the relative Picard scheme classifying line bundle along the fibers of $f$. Similarly, for any torus $T$, we define $\stPic_T(\frX/\frY)$ and $\Pic_T(\frX/\frY)$ to be the Picard stack and Picard scheme of $T$-torsors along the fibers of $f$.

For a flat morphism $f:\frX\to\frY$, we denote its relative dimension (the dimension of any geometric fiber) by $\dim(f)$ or, if the morphism is clear from the context, by $\dim(\frX/\frY)$.

\subsection{Notations concerning the curve $X$}\label{not:X}
For most part of these papers (except the Appendices on cohomological correspondences), we fix $X$ to be a smooth connected projective curve over $k$ of genus $g_X$. 

For any $k$-algebra $R$, let $\XR=\Spec R\times_{\Spec k} X$. If $x\in X(R)$ is an $R$-point of $X$, let $\Gamma(x)\subset\XR$ be the graph of $x:\Spec R\to X$. We let
\begin{equation}\label{eq:disk}
\disk_x=\Spec\hatO_x;\hspace{1cm}\pdisk_x=\Spec\hatO^{\punc}_x=\disk_x-\Gamma(x)
\end{equation}
be the spectra of the formal completion and the punctured formal completion of $\XR$ along $\Gamma(x)$. In general, for any scheme $S$ and any $x\in X(S)$, we also let $\Gamma(x)\subset S\times X$ be the graph of $x:S\to X$.

For a line bundle $\calL$ (or a divisor $D$) over $X$, we let $\rho_{\calL}$ (or $\rho_{D}$) denote the complement of the zero section in the total space of the line bundle $\calL$ (or $\calO_X(D)$). This is naturally a $\GG_m$-torsor.

For a divisor $D$ on $X$ and a quasi-coherent sheaf $\calF$ on $X$, we write $\calF(D):=\calF\otimes_{\calO_X}\calO_X(D)$.

Let $D$ be a divisor on $X$ and let $\frY$ be a stack over $X$ with a $\GG_m$-action such that the structure morphism $\frY\to X$ is $\GG_m$-invariant. Then we define
\begin{equation}\label{eq:twDX}
\frY_D:=\rho_D\twtimes{\GG_m}_X\frY=[(\rho_D\times_X\frY)/\GG_m].
\end{equation}
If $\frY$ is a stack over $k$, with no specified morphism to $X$ from the context, we also write
\begin{equation}\label{eq:twD}
\frY_D:=\rho_D\twtimes{\GG_m}\frY=[(\rho_D\times\frY)/\GG_m].
\end{equation}

\subsection{Notations concerning the reductive group $G$}\label{not:G}
Throughout these papers, we fix $G$ to be a connected reductive group over $k$ of semisimple rank $n$, and we fix a Borel subgroup $B$ of $G$ with universal quotient torus $T$. 

For each maximal torus in $B$ we get a Weyl group. If we use elements in $B$ to conjugate one torus to another, their Weyl groups are canonically identified. Let $W$ denote this canonical Weyl group. Similarly, the based root system $\Phi^+\subset\Phi\subset\xch(T)$ and the based coroot system $\Phi^{\vee,+}\subset\Phi^\vee\subset\xcoch(T)$ are canonically defined. Let $\ZZ\Phi^\vee\subset\xcoch(T)$ be the $\ZZ$-lattice spanned by the coroots $\Phi^\vee$.

Throughout the Thesis, we assume that char($k$) does not divide the order of $W$. Let
\begin{eqnarray*}
\Wa:=\ZZ\Phi^\vee\rtimes W,\\
\tilW:=\xcoch(T)\rtimes W
\end{eqnarray*}
be the {\em affine Weyl group} and the {\em extended affine Weyl group} of $G$.

Let $\frg,\frb,\frt$ be the Lie algebras of $G,B,T$ respectively. Let $\frc$ be the GIT quotient $\frg\sslash G=\frt\sslash W=\Spec k[f_1,\cdots,f_n]$. The affine space $\frc$ inherits a weighted $\GG_m$-action from $\frt$ such that $f_i$ is homogeneous of degree $d_i\in\ZZ_{\geq1}$.

Let $\frc^{\rs}$ be the regular semisimple locus of $\frc$, which is the complement of the discriminant divisor $\Delta\subset\frc$. For a stack $\frX$ or a morphism $F$ over $\frc$ or $\frc/\GG_m$, we use $\frX^{\rs}$ and $F^{\rs}$ to denote their restrictions to $\frc^{\rs}$ or $\frc^{\rs}/\GG_m$.

\subsection{Sheaf-theoretic notations}
For a Deligne-Mumford $\frX$, let $D^b_c(\frX)$ denote the derived category of constructible $\Ql$-complexes on $\frX$. We will consider both the usual $t$-structure on $D^b_c(\frX)$ and the perverse $t$-structure on $D^b_c(\frX)$. We follow the notation of \cite{BBD} concerning the perverse $t$-structures. For an object $\calF\in D^b_c(\frX)$, let $\tau_{\leq i}\calF$ and $\ptau_{\leq i}\calF$ be the truncations of $\calF$ under the usual and perverse $t$-structures; let $\bH^i\calF$ and $\pH^i\calF$ be the cohomology sheaves and perverse cohomology sheaves of $\calF$.

Let $f:\frX\to\frY$ be a morphism of Deligne-Mumford stacks. When the morphism $f$ is clear from the context, we will use $\homo{*}{\frX/\frY}$ to denote the homology complex $f_!f^!\const{\frY}$ on $\frY$, and use $\coho{*}{\frX/\frY}$ to denote the cohomology complex $f_*\const{\frX}$ on $\frY$. We also set
\begin{eqnarray*}
\homo{i}{\frX/\frY}&:=&\bH^{-i}(\homo{*}{\frX/\frY});\\
\coho{i}{\frX/\frY}&:=&\bH^i(\coho{*}{\frX/\frY})=\bR^if_*\const{\frX}.
\end{eqnarray*}
If $U\subset\frY$ is an open substack, we sometimes abuse the notation and simply write $\homo{*}{\frX/U}$, etc. to mean $\homo{*}{f^{-1}(U)/U}$, etc. If $\frY=\Spec k$, we simply write $\homog{*}{\frX}$ and $\cohog{*}{\frX}$ for the homology and cohomology groups of $\frX$.

For a morphism $f:\frX\to\frY$, we sometimes use $\DD_{\frX/\frY}$ or $\DD_f$ to denote the relative dualizing complex $f^!\Ql$. When $\frY=\Spec k$, we simply write $\DD_{\frX}$ for the dualizing complex of $\frX$, and we recall that the Borel-Moore homology (homology with closed support) is defined as
\begin{equation*}
\hBM{i}{\frX}:=\cohog{-i}{\frX,\DD_{\frX}}.
\end{equation*}

Let $\calF\mapsto\calF(1)$ be the usual Tate twist in $D^b_c(\frX)$. We sometimes need a half Tate twist $(1/2)$ only for notational convenience. In these situations, we formally add a square root of $\Ql(1)$ to the category $D^b_c(\frX)$.

%  Mpar

\section{The parabolic Hitchin fibration}\label{s:Mpar}
In this section, we introduce the main object of our study, the parabolic Hitchin moduli stack and the parabolic Hitchin fibration. We study their basic geometric properties, such as smoothness, flatness, product formula, etc., parallel to the study of the Hitchin fibration by Ng\^o in \cite{NgoFL}. Many proofs in this section are borrowed from their counterparts in {\em loc.cit.} with slight modification.

\subsection{The parabolic Hitchin moduli stack}\label{ss:parH}
In this subsection, we define the parabolic Hitchin moduli stack and the parabolic Hitchin fibration. We first recall that $\Bun_G$ is the moduli stack classifying $G$-torsor over the curve $X$. Let $\Bunpar_G$ be the moduli stack of $G$-torsors on $X$ with a $B$-reduction at a point. More precisely, for any scheme $S$, $\Bunpar_G(S)$ is the groupoid of triples $(x,\calE,\calE^B_x)$ where
\begin{itemize}
\item $x:S\to X$ with graph $\Gamma(x)$;
\item $\calE$ is a $G$-torsor over $S\times X$;
\item $\calE^B_x$ is a $B$-reduction of $\calE$ along $\Gamma(x)$.
\end{itemize}

\subsubsection{The parabolic Hitchin moduli stack}\label{sss:par}

Fix a divisor $D=2D'$ on $X$ (the assumption that $D$ is the square of another divisor is only used to construct a global Kostant section later, cf. \cite[2.2]{NgoFL} and (\ref{eq:kos})). Assume $\deg(D)\geq2g_X$ ($g_X$ is the genus of $X$).

Recall from \cite[Def. 4.2.1]{NgoFib} that the (generalized) {\em Hitchin moduli stack} $\MHit=\MHit_{G,X,D}$ is the functor which sends a scheme $S$ to the groupoid of {\em Hitchin pairs} $(\calE,\varphi)$ where
\begin{itemize}
\item $\calE$ is a $G$-torsor over $S\times X$;
\item $\varphi\in\cohog{0}{S\times X,\Ad(\calE)\otimes_{\calO_X}\calO_X(D)}$ is called a {\em Higgs field}.
\end{itemize}
Here $\Ad(\calE)=\calE\twtimes{G}\frg$ is the adjoint bundle associated to $\calE$.

\begin{defn}\label{def:par1}
The {\em parabolic Hitchin moduli stack} (or more precisely, {\em Hitchin moduli stack with parabolic structures}) $\Mpar=\Mpar_{G,X,D}$ is the
functor which sends a scheme $S$ to the groupoid of quadruples $(x,\calE,\varphi,\calE^{B}_x)$, where
\begin{itemize}
\item $x:S\to X$ with graph $\Gamma(x)$;
\item $(\calE,\varphi)\in\MHit_{G,X,D}(S)$ is a Hitchin pair;
\item $\calE^{B}_x$ is a $B$-reduction of $\calE$ along $\Gamma(x)$,
\end{itemize}
such that $\varphi$ is {\em compatible} with $\calE^B_x$, i.e.,
\begin{equation*}
\varphi|_{\Gamma(x)}\in \cohog{0}{\Gamma(x),\Ad(\calE^{B}_x)\otimes_{\calO_S}x^*\calO_X(D)}
\end{equation*}
\end{defn}

\begin{remark}
When part or all of the data $(G,X,D)$ is clear from the context, we will omit it from the subscripts of $\Mpar_{G,X,D}$.
\end{remark}

Forgetting the Higgs field $\varphi$ gives a morphism $\Mpar\to\Bun^{\parab}_G$ and forgetting the choice of the $B$-reduction gives a morphism
\begin{equation*}
\pi_{\calM}:\Mpar\to\MHit\times X.
\end{equation*}

We give a second (but equivalent) definition of $\Mpar$. By \cite[Sec. 4]{NgoFib}, the Hitchin moduli stack can be interpreted as classifying sections
\begin{equation*}
X\to [\frg/G]_D:=\rho_D\twtimes{\GG_m}[\frg/G].
\end{equation*}
Here $[\frg/G]$ is the adjoint quotient stack of $\frg$ by $G$, $\rho_D$ is the $\GG_m$-torsor over $X$ associated to the line bundle $\calO_X(D)$ (for the meaning of the twisting $(-)_D$ in general, see Sec. \ref{not:gen}). We have an evaluation morphism
\begin{equation*}
\ev^{\Hit}: \MHit\times X\to [\frg/G]_D.
\end{equation*}
Consider the $D$-twisted form of the Grothendieck simultaneous resolution:
\begin{equation*}
\pi:[\frb/B]_D\to[\frg/G]_D.
\end{equation*}

It it easy to see that

\begin{lemma}\label{l:par2}
The parabolic Hitchin moduli stack $\Mpar$ fits into a Cartesian square
\begin{equation}\label{d:2def}
\xymatrix{\Mpar\ar[rr]^{\ev^{\parab}}\ar[d]^{\pi_{\calM}} && [\frb/B]_D \ar[d]^{\pi}\\
          \MHit\times X\ar[rr]^{\ev^{\Hit}} && [\frg/G]_D}.
\end{equation}
\end{lemma}
We can take this to be an alternative definition of $\Mpar$.

\subsubsection{The Hitchin base}
Recall that the Hitchin base space $\AHit$ is the affine space of all sections of $X\to\frc_D$. Here the twisting $\frc_D:=\rho_D\twtimes{\GG_m}\frc$ uses the weighted $\GG_m$-action on $\frc$ inherited from the $\GG_m$-action on $\frt$ by homotheties. Fixing homogeneous generators $f_1,\cdots,f_n$ of the coordinate ring $\Sym_k(\frt^*)^W$ of $\frc$, we can write
\begin{equation}\label{eq:AHit}
\AHit=\bigoplus_{i=1}^n \cohog{0}{X,\calO_X(d_iD)}. 
\end{equation}
where $d_i$ is the degree of $f_i$.

The morphism $[\frg/G]_D\to\frc_D$ induces the {\em Hitchin fibration} (see \cite[4.1]{NgoFL})
\begin{equation}\label{eq:fHit}
f^{\Hit}:\MHit\to\AHit.
\end{equation}

\begin{defn}\label{def:Mparfib}
The morphism
\begin{eqnarray*}
\fpar:\Mpar&\to&\AHit\times X\\
(x,\calE,\varphi,\calE^B_x)&\mapsto&(f^{\Hit}(\calE,\varphi),x).
\end{eqnarray*}
is called the {\em parabolic Hitchin fibration}. The geometric fibers of the morphism $\fpar$ are called {\em parabolic Hitchin fibers}.
\end{defn}

\begin{defn}\label{def:unicam} The {\em universal cameral cover}, or the {\em enhanced Hitchin base} $\tcA$ is defined by the Cartesian diagram
\begin{equation}\label{d:tcA}
\xymatrix{\tcA\ar[r]^{\ev}\ar[d]^{q} & \frt_D\ar[d]^{q_{\frt}}\\
\AHit\times X\ar[r]^{\ev} & \frc_D}
\end{equation}
For $a\in\AHit(k)$, the pre-image $X_a:=q^{-1}(\{a\}\times X)$ is called the {\em cameral curve} $X_a$ corresponding to $a$. The projection $X_a\to X$ is denoted by $q_a$.
\end{defn}

The commutative diagram
\begin{equation}\label{d:gralt}
\xymatrix{[\frb/B]\ar[r]\ar[d]^{\pi} & \frt\ar[d]^{\sslash W}\\
[\frg/G]\ar[r]^{\chi} & \frc}
\end{equation}
together with the Lem. \ref{l:par2} gives a morphism
\begin{equation}\label{eq:tilf}
\tilf:\Mpar\to\tcA.
\end{equation}
which we call the {\em enhanced parabolic Hitchin fibration}.

\begin{lemma}\label{l:tcAsm}
Recall that $\deg(D)\geq2g_X$. Then $\tcA$ is smooth.
\end{lemma}
\begin{proof}
The affine space bundle $\frc_D$ over $X$ is non-canonically a direct sum of line bundles of the form $\calO(eD)$, each having degree $\geq2g_X$ since $e\geq1$. Therefore the line bundle $\calO(eD)$ is globally generated. Hence the evaluation map
\begin{equation*}
\AHit=\cohog{0}{X,\frc_D}\times X\to\frc_D
\end{equation*}
is a surjective bundle map over $X$ (the source is a trivial vector bundle), therefore it is smooth. This shows that the evaluation map $\AHit\times X\to \frc_D$ is smooth. Base change to $\frt_D$, we see that $\tcA\to\frt_D$ is smooth, hence $\tcA$ is smooth because $\frt_D$ is.
\end{proof}

We summarize the various stacks we considered into a commutative
diagram
\begin{equation}\label{d:allplayers}
\xymatrix{\Mpar\ar[r]^(.4){\nu_{\calM}}\ar[dd]^{\ev^{\parab}} & \MHit\times_{\AHit}\tcA\ar[dr]\ar[rr]^{q_{\calM}}\ar[dd] & & \MHit\times X\ar[dr]^{f^{\Hit}\times \id_X}\ar'[d][dd]^{\ev^{\Hit}} &\\
& & \tcA\ar[rr]^(.3){q}\ar[dd]^(.7){\ev} & & \AHit\times X\ar[dd]^(.7){\ev}\\
[\frb/B]_D\ar[r]^{\nu} & [\frg/G]_D\times_{\frc}\frt\ar'[r]^(.7){q_{\frg}}[rr]\ar[dr] & & [\frg/G]_D\ar[dr]^{\chi} &\\
& & \frt_D\ar[rr]^{q_{\frt}} & & \frc_D}
\end{equation}
where the three squares formed by horizontal and vertical arrows are Cartesian.

\begin{remark}\label{rm:variousbases}
We make a few remarks about the various open subsets of $\AHit$ that will appear in the sequel:
\begin{equation*}
\AHit\supset\Ah\supset\Aa\supset\calA\supset\calA_0\supset\Ag.
\end{equation*}

In \cite[Def. 4.4]{NgoFib}, Ng\^o introduced the open subscheme $\Ah\subset\AHit$ consisting of those $a:X\to\frc_D$ which {\em generically} lies in the regular semisimple locus $\frc^{\rs}_D\subset\frc_D$. In \cite[6.1]{NgoFL}, he introduced the open subset $\Aa$, which we will recall in Def. \ref{def:ani}. We will define the subset $\calA$ in Rem. \ref{rm:largedeg}, and we will mostly be working over $\calA$. If $\textup{char}(k)=0$, $\calA=\Aa$, and it is expected that they are equal in general. The open subset $\calA_0$ will be introduced in Sec. \ref{ss:codim}. The open subset $\Ag$ introduced in \cite[4.6]{NgoFL} consists of those sections $a:X\to\frc_D$ which are transversal to the discriminant divisor in $\frc_D$.
\end{remark}

\begin{exam}\label{ex:gl}
We describe the parabolic Hitchin stack and the parabolic Hitchin fibration in concrete terms when $G=\GL(n)$. The stack $\Mpar_{\GL(n)}$ classifies the data
\begin{equation*}
(x;\calE_0\supset\calE_1\supset\cdots\supset\calE_{n-1}\supset\calE_n=\calE_0(-x);\varphi)
\end{equation*}
where $\calE_i$ are rank $n$ vector bundles on $X$ such that $\calE_i/\calE_{i+1}$ are skyscraper sheaves of length 1 supported at $x\in X$. The Higgs field $\varphi$ in this case can be interpreted as a map of coherent sheaves
\begin{equation*}
\varphi:\calE_0\to\calE_0(D)
\end{equation*}
such that $\varphi(\calE_i)\subset\calE_i(D)$ for $i=1,\cdots,n-1$.

In this case, $\AHit$ is the space of characteristic polynomials
\begin{equation*}
\AHit=\bigoplus_{i=1}^n\cohog{0}{X,\calO_X(iD)}.
\end{equation*}
For each $a=(a_1,\cdots,a_n)\in\AHit$, we can define the {\em spectral curve} $p_a:Y_a\to X$, here $Y_a$ is an embedded curve in the total space of the line bundle $\calO_X(D)$ defined by the equation
\begin{equation*}
\sum_{i=0}^na_it^{n-i}=0
\end{equation*}
where $a_0=1$. Let $a\in\Ah(k)$, then $Y_a$ is a reduced curve. In this case, we have a natural isomorphism $\MHit_a\cong\overline{\stPic}(Y_a)$, the latter being the compactified Picard stack classifying torsion-free coherent sheaves on $Y_a$ of generic rank 1. This isomorphism sends $\calF\in\overline{\stPic}(Y_a)$ to $(p_{a,*}\calF,\varphi)$ where the Higgs field $\varphi$ on $p_{a,*}\calF$ comes from the action of a direct summand $\calO_X(-D)\subset\calO_{Y_a}$ on $\calF$. 

For $x\in X$, the parabolic Hitchin fiber $\Mpar_{a,x}$ classifies the data
\begin{equation*} 
\calF_0\supset\calF_1\supset\cdots\supset\calF_{n-1}\supset\calF_n=\calF_0(-x)
\end{equation*}
where $\calF_i\in\overline{\stPic}(Y_a)$ and each $\calF_i/\calF_{i+1}$ has length 1. If $Y_a$ is \'etale over $x$, then the above data amounts to the same thing as $\calF_0\in\overline{\stPic}(Y_a)$ together with an ordering of the set $p_a^{-1}(x)$.
\end{exam}

% P-action

\subsection{Symmetries on parabolic Hitchin fibers}\label{ss:sym}
In this subsection, we study the Picard stack action on the parabolic Hitchin stack.
 
\subsubsection{Facts about regular centralizers}
We first recall a few facts about the regular centralizer group schemes following \cite[Chapter 2]{NgoFL}.

Let $I_G\to\frg$ be the universal centralizer group scheme: for any $z\in\frg$, the fiber $I_{G,z}$ is the centralizer $G_z$ of $z$ in $G$. Let $\frg^{\reg}$ be the open subset of regular (i.e., centralizers have minimal dimension) elements in $\frg$. According to \cite[Lemme 2.1.1]{NgoFL}, there is a smooth group scheme $J\to\frc$, called the {\em group scheme of regular centralizers} together with an $\Ad(G)$-equivariant homomorphism
\begin{equation}\label{eq:Gcent}
\jmath_G:\chi^*J:=J_{\frg}\to I_G
\end{equation}
which is an isomorphism over $\frg^{\reg}$ (here $\chi:\frg\to\frc$ is the adjoint quotient). In other words, the stack $\BB J$ (over $\frc$) acts on the stack $[\frg/G]$, such that $[\frg^{\reg}/G]$ becomes a $J$-gerbe over $\frc$.

Let $J_{\frb}$ be the pull-back of $J$ to $\frb$. Let $I_B$ be the universal centralizer group scheme of the adjoint action of $B$ on $\frb$. We claim that

\begin{lemma}\label{l:bcent}
There is a natural homomorphism $\jmath_B:J_{\frb}\to I_B$ such that $\jmath_G|_{\frb}$ factors as $J_\frb\xrightarrow{\jmath_B} I_B\to I_G|_{\frb}$.
\end{lemma}
\begin{proof}
For $x\in\frb^{\reg}:=\frb\cap\frg^{\reg}$, we have canonical identifications $J_x=I_{G,x}=I_{B,x}$ (cf. \cite[Lemme 2.4.3]{NgoFL}). This gives the desired map $\jmath_B$ over $\frb^{\reg}$. To extend $\jmath_B$ to the whole $\frb$, we can use the same argument as in \cite[Prop. 3.2]{NgoFib}, because $J_{\frb}$ is smooth over $\frb$ and $\frb-\frb^{\reg}$ has codimension at least two.
\end{proof}
As in the situation of $J_{\frg}$, we can say that the stack $\BB J$ acts on the stack $[\frb/B]$ over $\frc$.

Let $J_{\frt}$ be the pull-back of $J$ along $\frt\to\frc$. Recall the following fact

\begin{lemma}[\cite{NgoFL}, Prop. 2.4.2]\label{l:mapJtoT}
There is a $W$-equivariant homomorphism of group schemes over $\frt$:
\begin{equation*}
\jmath_T:J_{\frt}\to T\times\frt
\end{equation*}
which is an isomorphism over $\frt^{\rs}$. Here the $W$-structure on $J_{\frt}=J\times_{\frc}\frt$ is given by its left action on $\frt$, while the $W$-structure on $T\times\frt$ is given by the diagonal left action.
\end{lemma}

Moreover, $J$ carries a natural $\GG_m$-action such that $J\to\frc$ is $\GG_m$-equivariant, therefore it makes sense to twist $J$ by the $\GG_m$-torsor $\rho_D$ over $X$ and get $J_D\to\frc_D$.

\subsubsection{The Picard stack $\calP$}
Recall from \cite[4.3.1]{NgoFL} that we have a Picard stack $g:\calP\to\AHit$ whose fiber $\calP_a$ over $a:S\to\AHit$ (viewed as a morphism $a:S\times X\to\frc_D$) classifies $J_a:=a^*J_D$-torsors on $S\times X$. According to \cite[Prop. 5.2]{NgoFib}, $\calP$ is smooth over $\Ah$. Since $\BB J_D$ acts on $[\frg/G]_D$, $\calP$ acts on $\MHit$ preserving the base $\AHit$.

There is an open dense substack $\MHitreg\subset\MHit$ parametrizing those $(E,\varphi):X\to[\frg/G]_D$ which land entirely in $[\frg^{\reg}/G]_D$. Since $[\frg^{\reg}/G]_D$ is a $\BB J_D$-gerbe over $\frc_D$, $\MHitreg$ is a $\calP$-torsor over $\AHit$.

Lem. \ref{l:bcent} implies that $\BB J_D$ acts on $[\frb/B]_D$, compatible with its action on $[\frg/G]_D$, and preserving the morphism $[\frb/B]_D\to\frt_D$. Using the moduli interpretations of $\Mpar$ and $\calP$, we get

\begin{lemma}\label{l:pretcA}
The Picard stack $\calP$ acts on $\Mpar$. The action is compatible with its action on $\MHit$ and preserves the enhanced parabolic Hitchin fibration $\tilf:\Mpar\to\tcA$. In other words, if we let $\tcP=\tcA\times_{\AHit}\calP$, viewed as a Picard stack over $\tcA$, then $\tcP$ acts on $\Mpar$ over $\tcA$.
\end{lemma}

We describe this action on the level of $S$-points. Let $(a,x)\in\AHit(S)\times X(S)$ and $(x,\calE,\varphi,\calE^B_{x})\in\Mpar(S)$ be a point over it. An $S$-point of $\calP$ is the same as a $J_a$-torsor $Q^J$ over $S\times X$. Then the effect of the $Q^J$-action is
\begin{equation}\label{eq:Pactionex}
Q^J\cdot(x,\calE,\varphi,\calE^B_{x})=(x,Q^J\twtimes{J_a,\jmath_G}(\calE,\varphi),Q^J\twtimes{J_a,\jmath_B}\calE^B_x).
\end{equation}

\begin{exam}\label{ex:glP}
We continue with Example \ref{ex:gl} of $G=\GL(n)$. For $a\in\Ah(k)$, $\calP_a=\stPic(Y_a)$ is the Picard stack of line bundles on $Y_a$. The action of $\calL\in\calP_a$ on $\Mpar_{a,x}$ is given by
\begin{equation*}
\calL\cdot(\calF_0\supset\calF_1\supset\cdots)=(\calL\otimes\calF_0\supset\calL\otimes\calF_1\supset\cdots)
\end{equation*}
where the tensor products are over $\calO_{Y_a}$. 
\end{exam}

\begin{lemma}\label{l:regreg} Let $\Mparreg\subset\Mpar$ be the preimage of $\MHitreg$ under the forgetful morphism $\Mpar\to\MHit$. Then $\Mparreg$ is a torsor under $\tcP$, and the morphism $\nu_{\calM}$ induces an isomorphism $\Mparreg\cong\MHitreg\times_{\calA}\tcA$.
\end{lemma}
\begin{proof}
Restricting the Cartesian diagram (\ref{d:2def}) to $\MHitreg\times X$, we get
\begin{equation}\label{eq:regCar}
\xymatrix{\Mparreg\ar[rr]^{\ev^{\parab}}\ar[d]^{\pi_{\calM}} && [\frb^{\reg}/B]_D \ar[d]^{\pi}\ar[r] & \frt_D\ar[d]^{q_\frt}\\
\MHitreg\times X\ar[rr]^{\ev^{\Hit}} && [\frg^{\reg}/G]_D\ar[r] & \frc_D}.
\end{equation}
Notice that the square on the RHS above is also Cartesian, by a well-known result of Kostant in \cite{Kos}. The outer Cartesian square of (\ref{eq:regCar}) implies that $\nu_{\calM}:\Mparreg\cong\MHitreg\times_{\calA}\tcA$ is an isomorphism. Since $\MHitreg$ is a $\calP$-torsor, $\Mparreg$ is a $\tcP$-torsor.
\end{proof}

\begin{cons}\label{cons:TtorsorP} We will construct a tautological $T$-torsor $\calQ^T$ over $\tcP$. For any $a\in\AHit(S)$ and any $J_a$-torsor $Q^J$ over $S\times X$, the pull-back $q_a^*Q^J$ is a $q_a^*J_a$-torsor over the cameral curve $X_a$. By Lem. \ref{l:mapJtoT}, we have a natural homomorphism of group schemes over $X_a$:
\begin{equation*}
\jmath_a:q_a^*J_a\to T\times X_a.
\end{equation*}
Therefore we can form the induced $T$-torsor $Q^T:=q_a^*Q^J\twtimes{q_a^*J_a,\jmath_a}T$ over $X_a$. Since $\jmath_a$ is $W$-equivariant, $Q^T$ carries a strong $W$-equivariant structure. Let $\stPic_{T}(\tcA/\AHit)^W$ be the stack of strong $W$-equivariant objects in $\stPic_{T}(\tcA/\AHit)$ (see Sec. \ref{not:gen} for the notation). The assignment $Q^J\mapsto Q^T$ gives a morphisms of Picard stacks over $\AHit$ and $\tcA$:
\begin{eqnarray}\label{eq:PtoWPic}
\jmath_{\calP}&:&\calP\to\stPic_{T}(\tcA/\AHit)^W;\\
\jmath_{\tcP}&:&\tcP\to\tcA\times_{\AHit}\stPic_T(\tcA/\AHit)^W.
\end{eqnarray}
Over $\tcA\times_{\AHit}\stPic_T(\tcA/\AHit)^W$ we have the usual Poincar\'e $T$-torsor $\Poin^T$. Then we define
\begin{equation*}
\calQ^T:=\jmath_{\tcP}^*\Poin^T.
\end{equation*}
\end{cons}

Note that $\Mpar$ also naturally sits over $\BB T$ via $\Mpar\xrightarrow{\ev^{\parab}}[\frb/B]_D\to\BB T$, which gives a $T$-torsor $\calL^T$ over $\Mpar$.

\begin{lemma}\label{l:tcPaction}
There is a natural 2-morphism making the following diagram commutative:
\begin{equation}\label{d:tcPact}
\xymatrix{\tcP\times_{\tcA}\Mpar\ar[rr]^{\act}\ar@<-4ex>[d]_{\calQ^T}\ar@<3ex>[d]^{\calL^T} & & \Mpar\ar[d]^{\calL^T}\\
\BB T\times\BB T\ar[rr]^{\textup{mult}} & & \BB T}
\end{equation} 
Here ``$\act$'' is the action map and ``\textup{mult}'' stands for the multiplication on the Picard stack $\BB T$ induced from the multiplication on $T$.
\end{lemma}
\begin{proof}
Let $a\in\AHit(S)$, $Q^J\in\calP_a(S)$ be a $J_a$-torsor over $S\times X$ and  $(x,\calE,\varphi,\calE^B_{x})\in\Mpar(S)$ be a point over $a$, which also gives a point $\tilx\in X_a(S)$. By Construction \ref{cons:TtorsorP}, the fiber of $\calQ^T$ over the point $(\tilx,Q^J)\in\tcP(S)$ is the $T$-torsor
\begin{equation*}
Q^T_{\tilx}:=\tilx^*Q^J\twtimes{\tilx^*J_a,\jmath}T
\end{equation*}
over $S=\Gamma(\tilx)\subset X_a$. Here $\jmath:\tilx^*J_a\to T$ is induced from $\jmath_T$.

On the other hand, The fiber of $\calL^T$ over the point $(x,\calE,\varphi,\calE^B_x)$ is the $T$-torsor $\calE^T_x:=\calE^B_{x}\twtimes{B}T$ over $\Gamma(x)$. By the description of the $\calP$-action on $\Mpar$ given in (\ref{eq:Pactionex}), after twisting by $Q^J$, the $T$-torsor $\calE^T_x$ becomes the $T$-torsor
\begin{equation*}
(\tilx^*Q^J\twtimes{\tilx^*J_a,\jmath_B}\calE^B_x)\twtimes{B}T\cong\tilx^*Q^J\twtimes{\tilx^*J_a,\jmath}\calE^T_x=Q^T_{\tilx}\twtimes{T}\calE^T_x.
\end{equation*}
which is precisely the product of $Q^T_{\tilx}$ (the fiber of $\calQ^T$ at $(\tilx,Q^J)$) and $\calE^T_x$ (the fiber of $\calL^T$ at $(x,\calE,\varphi,\calE^B_x)$) under the multiplication on $\BB T$. This completes the proof.
\end{proof}

Finally we recall the following definition.

\begin{defn}[\cite{NgoFL}, Lemme 6.1.1]\label{def:ani}
The {\em anisotropic} open subset $\Aa\subset\Ah$ is the open locus of $a\in\Ah(k)$ where $\pi_0(\calP_a)$ is finite.
\end{defn}
Over $\Aa$, $\calP$ is Deligne-Mumford and of finite type.

\begin{remark}
Since we work with a constant group scheme $G$ over $k$, the anisotropic locus $\Aa$ is nonempty if and only if $G$ is semisimple. In general, if the constant group $G$ is replaced by a quasi-split reductive group scheme $H$ over $X$, the anisotropic locus $\Aa$ is nonempty if and only if the split center of $H$ is trivial.
\end{remark}

% Local-global product formula

\subsection{The local-global product formula}\label{ss:locglob}
In this subsection, we will relate parabolic Hitchin fibers to products of affine Springer fibers.

\subsubsection{The local counterpart of $\Mpar$}
For each point $x\in X(k)$, denote the completed local ring of $X$ at $x$ and its field of fractions by $\hatO_x$ and $F_x$. Their spectra $\disk_x$ and $\pdisk_x$ are the formal disk and the punctured formal disk centered at $x$ (see notations in Sec. \ref{not:X}). Let $\gamma\in\frg(F_x)$. The Hitchin fibers have their local counterparts, the {\em affine Springer
fiber} $M^{\Hit}_x(\gamma)$ in the affine Grassmannian $\Grass_{G,x}=G(F_x)/G(\hatO_x)$. For the definition of $M^{\Hit}_{x}(\gamma)$, see \cite[3.2]{NgoFL}. 

Similarly, the parabolic Hitchin fibers have their local counterparts $M^{\parab}_x(\gamma)$ also called affine Springer fibers, which is a sub-ind-schemes of the affine flag variety $\Flag_{G,x}=G(F_x)/\bI_x$ ($\bI_x$ is the Iwahori subgroup of $G(\hatO_x)$ corresponding to $B$). The functor $M^{\parab}_x(\gamma)$ sends any scheme $S$ to the set of isomorphism classes of quadruples $(\calE,\varphi,\calE^B_x,\alpha)$ where
\begin{itemize}
\item $\calE$ is a $G$-torsor over $S\htimes\disk_x$;
\item $\varphi\in\cohog{0}{S\htimes\disk_x,\Ad(\calE)}$;
\item $\calE^B_x$ is a $B$-reduction along $S\times\{x\}$;
\item $\alpha$ is an isomorphism $(\calE,\varphi)|_{S\htimes\pdisk_x}\cong(\calE^{\triv},\gamma)$, where $\calE^{\triv}$ is the trivial $G$-torsor over $S\htimes\pdisk_x$.
\end{itemize}
For $\gamma\in\frg^{\rs}(F_x)$, the reduced structures $M^{\parab,\red}_x(\gamma)$ and $M^{\Hit,\red}_x(\gamma)$ are locally of finite type.

\subsubsection{The local counterpart of $\calP$} For $a\in\frc(\hatO_x)$, we also have the local counterpart $P_x(J_a)$ of the Picard stack $\calP$ over $\calA$. The functor $P_x(J_a)$ sends any scheme $S$ to the set of isomorphism classes of pairs $(Q^J,\tau)$ where
\begin{itemize}
\item $Q^J$ is a $J_a$ torsor over $S\htimes\disk_x$;
\item $\tau$ a trivialization of $Q^J$ over $S\htimes\pdisk_x$.
\end{itemize}

If $\chi(\gamma)=a$, then the group ind-scheme $P_x(J_a)$ acts on $M^{\Hit}_x(\gamma)$ and $M^{\parab}_x(\gamma)$.

Recall from \cite[2.2]{NgoFL} that if $D=2D'$, we have a global Kostant section 
\begin{equation}\label{eq:kos}
\epsilon:\AHit\to\MHit
\end{equation}
For $a\in\AHit(k)$, consider the Hitchin pair $\epsilon(a)=(\calE,\varphi)$. After trivializing $\calE|_{\disk_x}$ and choosing an isomorphism $\hatO_x(D_x)\cong\hatO_x$, we can identify $(\calE,\varphi)|_{\disk_x}$ with $(\calE^{\triv},\gamma_{a,x})$ for some element $\gamma_{a,x}\in\frg(\hatO_x)$ such that $\chi(\gamma)=a_x\in\frc(\hatO_x)$. Parallel to the product formula of Ng\^{o} in \cite[Prop. 4.13.1]{NgoFL}, we have

\begin{prop}[Product formula]\label{p:prod} Let $(a,x)\in\Ah(k)\times X(k)$ and let $U_a$ be the dense open subset $a^{-1}\frc^{\rs}_D$ of $X$. We have a {\em homeomorphism} of stacks:
\begin{equation}\label{eq:prod}
\calP_a\twtimes{P_x^{\red}(J_a)\times P'}\left(M^{\parab,\red}_x(\gamma_{a,x})\times M'\right)\to\Mpar_{a,x}.
\end{equation}
where
\begin{eqnarray*}
P'&=&\prod_{y\in X-U_a-\{x\}}P_y^{\red}(J_a);\\
M'&=&\prod_{y\in X-U_a-\{x\}}M^{\Hit,\red}_y(\gamma_{a,y}).
\end{eqnarray*}
\end{prop}

The proof is identical with that for Hitchin fibers (see \cite[Prop. 4.13.1]{NgoFL}).

\begin{cor}\label{c:dim}
For $a\in\Ah(k)$ and $\tilx\in X_a(k)$ with image $x\in X(k)$, we have the following equality of dimensions:
\begin{equation*}
\dim\Mpar_{a,\tilx}=\dim \Mpar_{a,x}=\dim \calP_a=\dim \MHit_a.
\end{equation*}
\end{cor}
\begin{proof}
By Kazhdan-Lusztig \cite[Sec. 4, Cor. 2]{KL88}, we have equalities for the dimension of affine Springer fibers
\begin{equation*}
\dim M^{\parab}_x(\gamma)=\dim P_x(J_a)=\dim M^{\Hit}_x(\gamma)
\end{equation*}
where $\chi(\gamma)=a\in\frc(\hatO_x)$. Now the required statement follows from Prop. \ref{p:prod}.
\end{proof}

% Smoothness

\subsection{Geometric properties of the parabolic Hitchin fibration}\label{ss:geom}
We first prove the smoothness of $\Mpar$.

\begin{prop}\label{p:smooth} Recall that $\deg(D)\geq2g_X$, then we have:
\begin{enumerate}
\item $\Mpar|_{\Ah}\to X$ is smooth;
\item $\Mpar|_{\Ah}$ is a smooth and equidimensional algebraic stack of dimension equal to $\dim\MHit+1$;
\item $\Mpar|_{\Aa}$ is a smooth Deligne-Mumford stack.
\end{enumerate}
\end{prop}
\begin{proof}
(1) To prove the smoothness, we do several steps of d\'{e}vissage and reduce to the proof of \cite[Th. 4.12.1]{NgoFL}. Let $R$ be an artinian local $k$-algebra and $I\subset R$ a square-zero ideal. Let $R_0=R/I$. Fix a point $x\in X(R)$ with image $x_0\in X(R_0)$ and a point $m_0=(x_0,\calE_0,\varphi_0,\calE^B_{x_0})\in\Mpar(R_0)$ over $x_0\in X(R_0)$. To establish the smoothness of $\Mpar\to X$, we need to lift this point to a point $(x,\calE,\varphi,\calE^B_{x})\in\Mpar(R)$. In other words, we have to find the dotted arrows in the following diagrams
\begin{equation*}
\hspace{-1cm}
\xymatrix{X_{R_0}\ar[r]^{m_0}\ar@{^{(}->}[d] & [\frg/G]_D\ar[d]\\
\XR\ar[r]^{pr_X}\ar@{.>}[ur]^{m} & X}
\hspace{3cm}
\xymatrix{\Spec R_0\ar[r]^{m_{x_0}}\ar@{^{(}->}[d] & [\frb/B]_D\ar[d]\\
\Spec R\ar@{.>}[ur]^{m_{x}}\ar[r]^{x} & X}
\end{equation*}
making the following diagram commutative:
\begin{equation}\label{d:cond}
\xymatrix{\Spec R\ar@{^{(}->}[d]^{(\id,x)}\ar@{.>}[r]^{m_x} & [\frb/B]_D\ar[d]^{\pi}\\
\XR\ar@{.>}[r]^{m} & [\frg/G]_D}
\end{equation}

For a morphism of stacks $\frX\to\frY$ which is unambiguous from the context, we write $L_{\frX/\frY}$ for its cotangent complex. The infinitesimal deformations $m$ of $m_0$ are controlled by the complex $\bR\Hom_{X_{R_0}}(m^*_0L_{[\frg/G]_D/X},I\otimes_k\calO_X)$; the infinitesimal deformations of $m_x$ of $m_{0,x}$ are controlled by the complex $\bR\Hom_{R_0}(m^*_{x_0}L_{[\frb/B]_D/X},I)$. The condition (\ref{d:cond}) implies that the infinitesimal liftings of $(x_0,\calE_0,\varphi_0,\calE^B_{x_0})\in\Mpar(R_0)$ to $\Mpar(R)$ over $x$ is controlled by the complex $\calK$ which fits into an exact triangle
\begin{eqnarray*}
\calK_I&\to&\bR\Hom_{X_{R_0}}(m^*_0L_{[\frg/G]_D/X},I\otimes_k\calO_X)\oplus\bR\Hom_{R_0}(m^*_{x_0}L_{[\frb/B]_D/X},I)\\
& \xrightarrow{(i^*,-\pi^*)} & \bR\Hom_{R_0}(i^*m_0^*L_{[\frg/G]_D/X},I)\to\calK_I[1]
\end{eqnarray*}
where $i$ is the embedding $\Gamma(x_0)\hookrightarrow X_{R_0}$ and $\pi^*$ is induced by natural map between cotangent complexes $\pi^*L_{[\frg/G]_D/X}\to L_{[\frb/B]_D/X}$. According to the calculation in \cite[4.12]{NgoFL}, we have
\begin{eqnarray*}
m^*_0L_{[\frg/G]_D/X}&=&[\Ad(\calE_0)^\vee\xrightarrow{\Ad(\varphi_0)}\Ad(\calE_0)^\vee(D)];\\
m^*_{x_0}L_{[\frb/B]_D/X}&=&[\Ad(\calE^B_{x_0})^\vee\xrightarrow{\Ad(\varphi_0(x_0))}\Ad(\calE^B_{x_0})^\vee\otimes i^*\calO_X(D)].
\end{eqnarray*}
Here the two-term complexes sit in degrees -1 and 0. we see that
\begin{equation*}
\calK_I\cong[\calF\otimes_{R_0}I\xrightarrow{\Ad(\varphi_0)}\calF(D)\otimes_{R_0}I]
\end{equation*}
where $\calF$ is the kernel of the surjective map
\begin{equation*}
\Ad(\calE_0)\twoheadrightarrow i_*i^*\Ad(\calE_0)\twoheadrightarrow i_*\left(i^*\Ad(\calE_0)/\Ad(\calE^B_{x_0})\right).
\end{equation*}
Since both the source and the target of the above surjection are flat over $R_0$, $\calF$ is also flat over $R_0$. Also, as a subsheaf of the vector bundle $\Ad(\calE_0)$ over $X_{R_0}$, $\calF$ is locally free over $\calO_X$.
 
The obstruction to the lifting lies in $\cohog{1}{X_{R_0},\calK_I}$. Writing $I$ as the quotient of a free $R_0$-module $R_0^{\oplus e}$ by a submodule $J$, we have a distinguished triangle
\begin{equation*}
\calK_J\to\calK_{R_0}^{\oplus e}\to\calK_I\to\calK_J[1].
\end{equation*}
The associated long exact sequence gives
\begin{equation*}
\cohog{1}{X_{R_0},\calK_{R_0}^{\oplus e}}\to \cohog{1}{X_{R_0},\calK_I}\to \cohog{2}{X_{R_0},\calK_J}=0.
\end{equation*}
Therefore it suffices to prove $\cohog{1}{X_{R_0},\calK}=0$ for $\calK=\calK_{R_0}=[\calF\xrightarrow{\Ad(\varphi_0)}\calF(D)]$. 

Let $\frm$ be the maximal ideal of $R_0$. Consider the $\frm$-adic filtration of the complex $\calK$ (which is finite since $R_0$ is artinian):
\begin{equation*}
\calK^n=[\frm^n\calF\xrightarrow{\Ad(\varphi_0)}\frm^n\calF(D)],\textup{for } n\geq0.
\end{equation*}
Since $\calF$ is flat over $R_0$, the associated graded pieces $\Gr^n\calK$ are of the form
\begin{equation}\label{eq:grk}
\Gr^n\calK=\frm^n/\frm^{n+1}\otimes[\calF/\frm\calF\xrightarrow{\Ad(\overline{\varphi_0})}\calF(D)/\frm\calF(D)]
\end{equation}
where $\overline{\varphi_0}$ is the reduction of $\varphi_0$ mod $\frm$.

To prove the vanishing of $\cohog{1}{X_{R_0},\calK}$, it suffices to prove the vanishing of each $\cohog{1}{X_{R_0},\Gr^n\calK}$. By the expression (\ref{eq:grk}), we may reduce to showing the vanishing of $\cohog{1}{X_{R_0},\calK}$ in the case $R_0$ is a field. We may assume $R_0=k$ otherwise make a base change for $X$ from $k$ to $R_0$ and the argument is the same. In this case, as in the proof of \cite[Th. 4.12.1]{NgoFL}, using Serre duality, we reduce to showing that $\cohog{0}{X,\calK'}=0$ where
\begin{equation*}
\calK'=\ker(\calF^\vee(-D)\otimes\omega_X\xrightarrow{\Ad(\varphi_0)}\calF^\vee\otimes\omega_X)
\end{equation*}
Here $\calF^\vee=\bR\underline{\Hom}_{X}(\calF,\calO_X)$ is a subsheaf of $\Ad(\calE_0)^\vee(x_0)$ (the sheaf of rational sections of $\Ad(\calE_0)^\vee$ with at most a simple pole at $x_0$). Therefore $\calK'$ is a subsheaf of
\begin{equation*}
\calK''=\ker(\Ad(\calE_0)^\vee(-D+x_0)\otimes\omega_X\xrightarrow{\Ad(\varphi_0)}\Ad(\calE_0)^\vee(x_0)\otimes\omega_X).
\end{equation*}
In \cite[Th. 4.12.1]{NgoFL}, Ng\^{o} proves that $\cohog{0}{X,\calK''}$, as the obstruction to the lifting problem for the Hitchin moduli stack $\MHit_{G,X,D-x_0}$, vanishes whenever $\deg(D-x_0)>2g_X-2$. In our case, $\deg(D)\geq2g_X$ so this condition holds, therefore $\cohog{0}{X,\calK''}=0$, hence $\cohog{0}{X,\calK'}=0$. This proves the vanishing of the obstruction group $\cohog{1}{X_{R_0},\calK_I}$ in general, and completes the proof of smoothness of $\Mpar\to X$.

(2) The relative dimension of $\Mpar\to X$ at a $k$-point $(x_0,\calE_0,\varphi_0,\calE^B_{x_0})$ is the Euler characteristic $\chi(X,\calK)$ of the complex $\cohog{*}{X,\calK}$. Recall that $\calK$ fits into the distinguished triangle
\begin{equation*}
\calK\to\tcK\to\calQ\to\calK[1]
\end{equation*}
where
\begin{eqnarray*}
\tcK&=&[\Ad(\calE_0)\xrightarrow{\Ad(\varphi_0)}\Ad(\calE_0)(D)];\\
\calQ&=&i_*[i^*\Ad(\calE_0)/\Ad(\calE^B_{x_0})\xrightarrow{\Ad(\varphi_0)}i^*\Ad(\calE_0)/\Ad(\calE^B_{x_0})\otimes i^*\calO_X(D)].
\end{eqnarray*}
It is clear that $\chi(X,\calQ)=0$. Therefore $\chi(X,\calK)=\chi(X,\tcK)$. But by the calculation in \cite[4.12]{NgoFL}, $\chi(X,\tcK)$ is the dimension of $\MHit_{G,X,D}$ at the $k$-point $(\calE_0,\varphi_0)$. Since $\MHit$ is equidimensional by \cite[4.4.6]{NgoFL}, $\Mpar\to X$ is also equidimensional of relative dimension equal to $\dim\MHit$.

(3) By \cite[Prop. 6.2.2]{NgoFL}, $\MHit|_{\Aa}$ is Deligne-Mumford. By Lem. \ref{l:par2}, $\pi_{\calM}:\Mpar\to\MHit\times X$ is schematic and of finite type because the Grothendieck simultaneous resolution $\pi$ is. Therefore $\Mpar|_{\Aa}$ is Deligne-Mumford.
\end{proof}

\begin{cor}\label{c:flat}
The parabolic Hitchin fibrations $\fpar:\Mpar|_{\Ah}\to\Ah\times X$ and $\tilf:\Mpar|_{\Ah}\to\tcA$ are flat. The fibers are local complete intersections. The restrictions of $\fpar$ and $\tilf$ to $\Aa$ are proper.
\end{cor}
\begin{proof} 
The source and the target of the maps $\fpar$ and $\tilf$ are smooth and equidimensional. In both cases, the relative dimensions are equal to $\dim\MHit-\dim\AHit$. By Cor. \ref{c:dim}, the dimension of the fibers of $\fpar$ and $\tilf$ are also equal to the dimension of Hitchin fibers, which, in turn, equals to $\dim\MHit-\dim\AHit$ by \cite[Prop. 4.14.4]{NgoFL}. Therefore by \cite[Ch. III, Exercise 10.9]{Hart}, both maps are flat and the fibers are local complete intersections.

By \cite[Prop. 6.1.5]{NgoFL}, $\MHit|_{\Aa}$ is proper over $\Aa$. By Lem. \ref{l:par2}, $\Mpar|_{\Aa}$ is proper over $\MHit|_{\Aa}\times X$ and $\MHit|_{\Aa}\times_{\Aa}\tcA^{\ani}$ because they are obtained by base change from the proper maps $[\frb/B]_D\to[\frg/G]_D$ and $[\frb/B]_D\to[\frg/G]_D\times_{\frc_D}\frt_D$. Therefore $\Mpar|_{\Aa}$ is proper over $\Aa\times X$ and $\tcA^{\ani}$.
\end{proof}

\begin{cor}
The pre-image $\Mpar_a$ of any $a\in\Ah(k)$ under the map $\Mpar\to\calA$ is reduced and flat over $X$.
\end{cor}
\begin{proof}
The flatness over $X$ follows by base changing $\fpar$ to $\{a\}\times X$. Over the open dense locus $U_a=a^{-1}(\frc^{\rs}_D)\subset X$, $q^{-1}(U_a)$ is smooth, and $\Mpar_a|_{U_a}\cong\MHit_a\times q^{-1}(U_a)$. Since $\MHit_a$ is reduced by \cite[Corollaire 4.14.4]{NgoFL}, $\Mpar_a|_{U_a}$ is also reduced. By flatness, $\Mpar_a|_{U_a}$ is dense in $\Mpar_a$. Moreover, $\Mpar_a$ is a local complete intersection by Cor. \ref{c:flat}, hence it does not admit embedded components. Therefore $\Mpar_a$ is also reduced.
\end{proof}

\begin{remark} The parabolic Hitchin fibers $\Mpar_{a,x}$ are not reduced in general. For example, suppose $q_a:X_a\to X$ is ramified over $x$ and $q^{-1}_a(x)$ consists of smooth points of $X_a$, then $\Mpar_{a,x}=\MHit_a\times q^{-1}_a(x)$ is not reduced because $q^{-1}_a(x)$ is not.
\end{remark}

\begin{remark} Although $\calP_a$ still has an open orbit in $\Mparreg_{a,\tilx}\subset\Mpar_{a,\tilx}$ (where $\tilx\in X_a$) which is a torsor under
$\calP_a$, this open set may not be dense in $\Mpar_{a,\tilx}$: there might be other irreducible components, as we will see in \cite[Sec. 4.1]{GSIII}.
\end{remark}

% codim estimate

\subsection{Stratification by $\delta$ and codimension estimate}\label{ss:codim}

\subsubsection{The global $\delta$} Recall that in \cite[3.7 and 4.9]{NgoFL}, Ng\^{o} introduced local and global Serre invariants $\delta$. The global invariant is the dimension of the affine part of the Picard stack $\calP_a$. It defines an upper semi-continuous function
\begin{equation}\label{eq:defglobdel}
\delta:\Aa\to\ZZ_{\geq0}.
\end{equation}
For an integer $\delta\geq0$, let $\Aa_{\delta}$ be the $\delta$-level set of $\Aa$.

\begin{lemma}\label{l:regwhole}
The open subset $\Aa_0$ of $\Aa$ consists precisely of points
$a\in\Aa$ where $\MHitreg_a=\MHit_a$.
\end{lemma}
\begin{proof}
When $\delta(a)=0$, each component of the Picard stack $\calP_a$ is an abelian stack, and $\pi_0(\calP_a)$ is finite, hence $\calP_a$ is proper. Therefore $\MHitreg_a$, being a torsor under $\calP_a$, is also proper. Since $\MHitreg_a$ is open dense in $\MHit_a$, we must have $\MHitreg_a=\MHit_a$.

Conversely, if $\MHitreg_a=\MHit_a$, then $\MHitreg_a$ is proper because $\MHit_a$ is proper. Since $\MHitreg_a$ is a $\calP_a$-torsor, $\calP_a$ is also proper. Hence the affine part of $\calP_a$ is zero-dimensional, i.e., $\delta(a)=0$.
\end{proof}

\subsubsection{The local $\delta$} The local Serre invariant (\cite[3.7]{NgoFL}) assigns to every $(a,x)\in(\Aa\times X)(k)$ the dimension of the corresponding affine Springer fiber $M^{\Hit}_x(\gamma_{a,x})$ (see the discussion before Prop. \ref{p:prod} for the meaning of $\gamma_{a,x}$), or, equivalently, the dimension of the local symmetry group $P_x(J_a)$. It defines an upper semi-continuous function
\begin{equation}\label{eq:deflocdel}
\delta:\Aa\times X\to\ZZ_{\geq0}
\end{equation}
For an integer $\delta\geq0$, let $(\Aa\times X)_{\delta}$ be the $\delta$-level set. For any $a\in\Aa(k)$, we have
\begin{equation*}
\delta(a)=\sum_{x\in X(k)}\delta(a,x).
\end{equation*}

\begin{lemma}\label{l:nuisom}
The open subset $(\Aa\times X)_0$ is precisely the locus where $\nu_{\calM}:\Mpar|_{\Aa}\to\MHit|_{\Aa}\times_{\Aa}\tcA^{\ani}$ is an isomorphism.
\end{lemma}
\begin{proof}
We need to check that for a geometric point $(a,x)\in\Aa(k)\times X(k)$, $\nu_{\calM}$ is an isomorphism over $(a,x)$ if and only if $\delta(a,x)=0$. By the left-most Cartesian square of the diagram (\ref{d:allplayers}), $\nu_{\calM}$ is an isomorphism over a geometric point $(\calE,\varphi,x)\in\MHit\times X$ if and only if $\nu:[\frb/B]\to[\frg/G]\times_{\frc}\frt$ is an isomorphism over $\varphi(x)\in\frg$. This is equivalent to saying that $\varphi(x)\in\frg^{\reg}$, by the result of Kostant (cf. \cite{Kos}).

Now for $(a,x)\in\Aa\times X$, $\nu_{\calM}$ is an isomorphism over $(a,x)$ if and only if $\nu_{\calM}$ is an isomorphism over all $(\calE,\varphi,x)\in\MHit_a\times X$. The above argument shows that this is equivalent to requiring $\ev_x:\MHit_a\to[\frg/G]_D$ to land in $[\frg^{\reg}/G]$. Let $\gamma_{a,x}\in\frg(\hatO_x)$ be chosen as in Prop. \ref{p:prod}. Recall from \cite[3.3]{NgoFL} that we have an open sub-ind-scheme $M^{\Hit,\reg}_x(\gamma_{a,x})\subset M^{\Hit}_x(\gamma_{a,x})$ defined in a similar way as $\MHitreg_a$, and which is a torsor under $P_x(J_a)$. The above discussion implies that $\nu_{\calM}$ is an isomorphism over $(a,x)$ if and only if $M^{\Hit}_x(\gamma_{a,x})=M^{\Hit,\reg}_x(\gamma_{a,x})$. We have to show that this condition is equivalent to $\delta(a,x)=0$.

If $\delta(a,x)=0$, then $M^{\Hit}_x(\gamma_{a,x})=M^{\Hit,\reg}_x(\gamma_{a,x})$ by \cite[Corollaire 3.7.2]{NgoFL}. 

Conversely, suppose $M^{\Hit}_x(\gamma)=M^{\Hit,\reg}_x(\gamma)$. According to \cite[4.13]{NgoFL} and \cite{KL88}, there is a lattice subgroup $\Lambda\subset P_x(J_a)$ which acts freely on $M^{\Hit}_x(\gamma)$ such that $\Lambda\backslash M^{\Hit,\red}_x(\gamma)$ is a projective variety and $\Lambda\backslash P_x^{\red}(J_a)$ is an affine group scheme. Therefore $\Lambda\backslash M^{\Hit,\red}_x(\gamma)=\Lambda\backslash M^{\Hit,\reg,\red}_x(\gamma)$ is both proper and affine since the latter is a torsor under $\Lambda\backslash P_x^{\red}(J_a)$. Hence they must be zero-dimensional, i.e., $\delta(a,x)=0$. This completes the proof.
\end{proof}

% estimate

Recall the following codimension estimate:

\begin{prop}[Ng\^{o} \cite{NgoFL}, Prop. 5.4.2]\label{p:globaldel}
\begin{enumerate}
\item []
\item If $\textup{char}(k)=0$, then the codimension of $\Aa_{\delta}$ in $\Aa$ is at least $\delta$.
\item If $\textup{char}(k)>0$, then for any $\delta_0$, there exists a positive integer $N(\delta_0)$ such that whenever $\deg(D)\geq N(\delta_0)$ and $\delta\leq\delta_0$, the codimension of $\Aa_{\delta}$ in $\Aa$ is at least $\delta$.
\end{enumerate}
\end{prop}

\begin{remark}\label{rm:largedeg}
Most of the results in this series of papers will depend on this estimate. {\em From now on, we fix an open subset $\calA\subset\Aa$ on which the estimate $\codim_{\calA}(\calA\cap\Aa_\delta)\geq\delta$ holds for any $\delta\in\ZZ_{\geq0}$.} According to Prop. \ref{p:globaldel}, if $\textup{char}(k)=0$, we can take $\calA=\Aa$. In the case $\textup{char}(k)>0$, we can for example take $\calA$ to be the union of $\Aa_\delta$ for all $\delta\leq\delta_0$, as long as $\deg(D)\geq N(\delta_0)$. In fact, it is expected that the estimate $\codim(\Aa_\delta)\geq\delta$ holds for any $\delta$ regardless of $\textup{char}(k)$. If this was proved, we could always take $\calA=\Aa$.

We use $\calA_\delta$ and $(\calA\times X)_\delta$ to denote the level sets of the global and local Serre invariants on $\calA$ and $\calA\times X$.
\end{remark}

We have an easy consequence of the above codimension estimate.

\begin{cor}\label{c:localdel}
For $\delta\geq1$, the codimension of $(\calA\times X)_{\delta}$ in $\calA\times X$ is at least $\delta+1$.
\end{cor}
\begin{proof}
For any $a\in\Aa(k)$ there are only finitely many $x\in X(k)$ such that $\delta(a,x)>0$. In other words, for $\delta\geq1$, the projection $(\calA\times X)_{\delta}\to\calA_{\delta}$ is quasi-finite. Hence
\begin{equation*}
\dim(\calA\times X)-\dim(\calA\times X)_{\delta}\geq\dim(\calA\times X)-\dim(\calA_{\delta})\geq\delta+1.
\qedhere
\end{equation*}
\end{proof}

%  affine Weyl action

\section{Construction of the affine Weyl group action}\label{s:Waff}
We first review two constructions of the classical Springer representations, one using middle extension of perverse sheaves and the other using Steinberg correspondences. It turns out that the second construction generalizes to the global situation. We introduce Hecke correspondences in the context of parabolic Hitchin fibration, which play the role of Steinberg correspondences in the classical theory. We then use them to construct the affine Weyl group action on the parabolic Hitchin complex.

In this section, we will always restrict to the open subset $\calA\subset\Aa$ fixed as in Rem. \ref{rm:largedeg}. {\em All stacks over $\AHit$ will be automatically restricted to $\calA$, without changing of notation.}

For a stack $\frX$ or a morphism $F$ over $\calA\times X$, we use $\frX^0$ and $F^0$ to denote their restrictions on the open subset $(\calA\times X)_0\subset\calA\times X$ (the locus where $\delta(a,x)=0$, see (\ref{eq:deflocdel})). For example, we have $\tcA\supset\tcA^0\supset\tcArs$.

\subsection{Review of classical Springer representations}\label{ss:claction}
In this subsection, we recall two constructions of the Weyl group action on the direct image complex for the Grothendieck simultaneous resolution $\pi:\tilg=\frb\twtimes{B}G\to\frg$. Note that over the regular semisimple locus $\frc^{\rs}$, we have a Cartesian diagram
\begin{equation}\label{eq:GrSp}
\xymatrix{\tilg^{\rs}\ar[r]\ar[d]^{\pi^{\rs}} & \frt^{\rs}\ar[d]^{q^{\rs}_{\frt}}\\
\frg^{\rs}\ar[r]^{\chi} & \frc^{\rs}}
\end{equation}
The left action of $W$ on $\frt$ gives a right action: $w$ acts by $t\cdot w=\Ad(w)^{-1}t$. Under this right $W$-action, $\pi^{\rs}$ and $q^{\rs}_{\frt}$ both become right $W$-torsors.

The first construction uses middle extension of perverse sheaves.

\begin{cons}[see \cite{L81},\cite{BM}]\label{cons:claction1}
Since $\pi^{\rs}$ realizes $\tilg^{\rs}$ as a right $W$-torsor over $\frg^{\rs}$, we get a left action of $W$ on $\pi_*\Ql|_{\frg^{\rs}}$ given by pulling back along the action maps. It is well-known that $\pi$ is a small map, so that $\pi_*\Ql$ is the middle extension (after shift to the perverse degree) from its restriction on $\frg^{\rs}$.  Therefore, this left $W$-action uniquely extends to a left $W$-action on $\pi_*\Ql$.
\end{cons}

The second construction uses cohomological correspondences.
\begin{cons}[cf. \cite{CG} for a treatment over $\CC$]\label{cons:claction2}
Consider the self-product
\begin{equation*}
[\frb/B]\times_{[\frg/G]}[\frb/B]=[\St/G]
\end{equation*}
where $\St=\tilg\times_{\frg}\tilg$ is the {\em Steinberg variety of triples}. We view $\St$ as a correspondence
\begin{equation*}
\xymatrix{ & \St\ar[dl]_{\overleftarrow{s}}\ar[dr]^{\overrightarrow{s}} & \\
\tilg\ar[dr]_{\pi} & & \widetilde{\frg}\ar[dl]^{\pi}\\
& \frg & }
\end{equation*}
With respect to the open subset $U=\frg^{\rs}$, $\St$ satisfies the condition (G-2) in Def. \ref{def:grlike} (in fact, this is equivalent to saying that $\pi$ is a small map). Moreover, we have a natural map $\mu:\St*\St\to\St$ given by forgetting the middle $\widetilde{\frg}$, which is obviously associative. By the discussion in Sec. \ref{ss:convalg}, $\Corr(\St;\Ql,\Ql)$ and $\Corr(\St^{\rs};\Ql,\Ql)$ have natural algebra structures given by convolutions. The top-dimensional irreducible components of $\St$ are indexed by $w\in W$, such that $\St_w^{\rs}$ is the graph of the right $w$-action on $\widetilde{\frg}^{\rs}$, i.e., $\St_w^{\rs}$ consists of points $(x,x\cdot w)=(x,w^{-1}x)$ for $x\in\tilg^{\rs}$. Therefore, we have an algebra isomorphism
\begin{equation*} 
\xymatrix{\Ql[W]\ar[r]^{\sim} & \cohog{0}{\St^{\rs},\DD_{\overleftarrow{s}}}\ar[d]^{\wr} & \cohog{0}{\St,\DD_{\overleftarrow{s}}}\ar[l]_{\sim}\ar[d]^{\wr}\\
& \Corr(\St^{\rs};\Ql,\Ql) & \Corr(\St;\Ql,\Ql)\ar[l]_{\sim}}
\end{equation*}
where $\Ql[W]$ is the group algebra of $W$. Under this isomorphism, $w$ corresponds to the fundamental class $[\St_{w}]$. By Prop. \ref{p:alghom}, we get an algebra homomorphism
\begin{equation*}
\Ql[W]\cong\Corr(\St;\Ql,\Ql)\xrightarrow{(-)_\#}\End_{\frg}(\pi_*\Ql)
\end{equation*}
sending $w$ to $[\St_{w}]_\#$. In other words, we get a left $W$-action on $\pi_*\Ql$. 
\end{cons}

\begin{lemma}\label{l:clsame}
Constructions \ref{cons:claction1} and \ref{cons:claction2} give the same $W$-action on $\pi_*\Ql$.
\end{lemma}
\begin{proof}
Over $\frg^{\rs}$, the two constructions are obviously the same because $\St^{\rs}_w$ is the graph of the right $w$-action on $\tilg^{\rs}$ (cf. Example \ref{ex:graph}). Since $\pi_*\Ql$ is the middle extension (up to shift) of its restriction on $\frg^{\rs}$, the restriction map
\begin{equation*}
\End_{\frg}(\pi_*\Ql)\to\End_{\frg^{\rs}}(\pi_*\Ql|_{\frg^{\rs}})
\end{equation*}
is an isomorphism. Therefore, the two constructions yield the same action over $\frg$.
\end{proof}

% Hecke correspondence

\subsection{The Hecke correspondence}\label{ss:Hecke}

We first recall the definition of the {\em Hecke correspondence} between $\Bunpar_G$ and itself over $X$:
\begin{equation}\label{eq:HkBun}
\xymatrix{ & \Hecke^{\Bun} \ar[dl]_{\overleftarrow{b}}\ar[dr]^{\overrightarrow{b}} & \\
 \Bunpar_G\ar[dr]_{} & & \Bunpar_G\ar[dl]^{}\\
& X & }
\end{equation}
For any scheme $S$, $\Hecke^{\Bun}(S)$ classifies tuples $(x,\calE_1,\calE^B_{1,x},\calE_2,\calE^B_{2,x},\alpha)$ where
\begin{itemize}
\item $(x,\calE_i,\calE^B_{i,x})\in\Bunpar_G(S)$ for $i=1,2$; 
\item $\alpha:\calE_1|_{\XS-\Gamma(x)}\isom\calE_2|_{\XS-\Gamma(x)}$ is an isomorphism of $G$-torsors.
\end{itemize}
It is well-known that $\Hecke^{\Bun}$ has a stratification
\begin{equation*}
\Hecke^{\Bun}=\bigsqcup_{\tilw\in\tilW}\Hecke^{\Bun}_{\tilw}.
\end{equation*}
The Bruhat order on $\tilW$ coincides with the partial order induced by the closure relation among the strata $\Hecke^{\Bun}_{\tilw}$.

\begin{defn}\label{def:Heckep} The {\em Hecke correspondence} $\Heckep$ is a self-correspondence of $\Mpar$ over $\calA\times X$:
\begin{equation*}
\xymatrix{ & \Heckep \ar[dl]_{\overleftarrow{h}}\ar[dr]^{\overrightarrow{h}} & \\
 \Mpar\ar[dr]_{\fpar} & & \Mpar\ar[dl]^{f^{\parab}}\\
& \calA\times X & }
\end{equation*}
For any scheme $S$, $\Heckep(S)$ classifies tuples $(x,\calE_1,\varphi_1,\calE^{B}_{1,x},\calE_2,\varphi_2,\calE^{B}_{2,x},\alpha)$ where
\begin{itemize}
\item $(x,\calE_i,\varphi_i,\calE^{B}_{i,x})\in\Mpar(S)$ for $i=1,2$;
\item $\alpha$ is an isomorphism of Hitchin pairs $(\calE_1,\varphi_1)|_{\XS-\Gamma(x)}\isom(\calE_2,\varphi_2)|_{\XS-\Gamma(x)}$.
\end{itemize}
\end{defn}

By definition, we have a commutative diagram of correspondences
\begin{equation}\label{d:hhecke}
\xymatrix{\Heckep\ar@<-1ex>@/_/[d]_{\overleftarrow{h}}\ar@<1ex>@/^/[d]^{\overrightarrow{h}}\ar[r]^{\beta} & \Hecke^{\Bun}\ar@<-1ex>@/_/[d]_{\overleftarrow{b}}\ar@<1ex>@/^/[d]^{\overrightarrow{b}}\\
\Mpar\ar[d]^{\fpar}\ar[r] & \Bunpar_G\ar[d]\\
\calA\times X\ar[r] & X}
\end{equation}

\begin{lemma}\label{l:heckeft} The functor $\Heckep$ is representable by an ind-algebraic stack of ind-finite type, and the two projections $\overleftarrow{h},\overrightarrow{h}:\Heckep\to\Mpar$ are ind-proper.
\end{lemma}
\begin{proof}
From the definitions, we see that the fibers of $\overleftarrow{h},\overrightarrow{h}$ are closed sub-ind-schemes of the fibers of $\overleftarrow{b},\overrightarrow{b}$. Hence it suffices to check the same statement for $\Hecke^{\Bun}$.

For each $\tilw\in\tilW$, let $\Hecke^{\Bun}_{\leq\tilw}$ be the closure of $\Hecke^{\Bun}_{\tilw}$. Then the projections
\begin{equation*}
\overleftarrow{b},\overrightarrow{b}:\Hecke^{\Bun}_{\leq\tilw}\to\Bunpar_{G}
\end{equation*}
are \'etale locally trivial bundles with fibers isomorphic to Schubert cycles in the affine flag variety $\Flag_G$. In particular, $\Hecke^{\Bun}_{\leq\tilw}$ is proper over $\Bunpar_G$ for both projections. This proves the lemma.
\end{proof}

Let us describe the fibers of $\overleftarrow{h}$ and $\overrightarrow{h}$. Let $(x,\calE,\varphi,\calE^B_{x})\in\Mpar(k)$. After trivializing $\calE|_{\Delta_x}$ and choosing an isomorphism $\hatO_x\cong\hatO_x(D)$, we get an isomorphism $\tau:(\calE,\varphi)|_{\disk_x}\isom(\calE^{\triv},\gamma_{a,x})$ for some $\gamma_{a,x}\in\frg(\hatO_x)$ such that $\chi(\gamma_{a,x})=a\in\frc(\hatO_x)$.

\begin{lemma}\label{l:hfiber}
The fibers of $\overleftarrow{h}$ and $\overrightarrow{h}$ over $(x,\calE,\varphi,\calE^B_{x})\in\Mpar(k)$ are isomorphic to the affine Springer fiber $M^{\parab}_x(\gamma_{a,x})$.
\end{lemma}
\begin{proof}
We prove the statement for $\overrightarrow{h}$; the other one is similar. For any scheme $S$, we have a natural map
\begin{eqnarray}\label{eq:afffib}
\overrightarrow{h}^{-1}(x,\calE,\varphi,\calE^B_{x})(S)&\to & M^{\parab}_x(\gamma_{a,x})(S)\\
\notag
(\calE_1,\varphi_1,\calE^B_{1,x},\alpha)&\mapsto& (\calE_1|_{\disk_x},\varphi_1|_{\disk_x},\calE^B_{1,x},\beta)
\end{eqnarray}
where $\alpha:(\calE_1,\varphi_1)|_{(X-\{x\})\times S}\isom(\calE,\varphi)|_{(X-\{x\})\times S}$ and 
\begin{equation*}
\beta:(\calE_1,\varphi_1)|_{\pdisk_x\htimes S}\xrightarrow{\alpha}(\calE,\varphi)|_{\pdisk_x\htimes S}\xrightarrow{\tau}(\calE^{\triv},\gamma_{a,x}).
\end{equation*} 
Now it is easy to see that (\ref{eq:afffib}) is injective; it is also surjective because any local modification of $(\calE^{\triv},\gamma_{a,x})$ on $\pdisk_x\htimes S$ can be glued with $(\calE,\varphi)|_{(X-\{x\})\times S}$ to get a Hitchin pair on $X$. Therefore (\ref{eq:afffib}) induces is an isomorphism for any $S$.
\end{proof}

In \cite[Sec. 3.4]{GSII}, we will study the relation between $\Hecke^{\Bun}$ and $\Heckep$.

% Hecke over nice locus

\subsection{Hecke correspondence over the nice locus}\label{ss:heckenice}
In this subsection, we determine the structure of the Hecke correspondence $\Heckep$ over the locus $(\calA\times X)_0$.

Consider the map 
\begin{equation}\label{eq:wpart}
\Heckep\xrightarrow{\overleftrightarrow{h}}\Mpar\times_{\calA\times X}\Mpar\xrightarrow{(\tilf,\tilf)}\tcA\times_{\calA\times X}\tcA\to\frt_D\times_{\frc_D}\frt_D.
\end{equation}
For $w\in W$, let $\frt_{w,D}\subset\frt_{D}\times_{\frc_D}\frt_D$ be the graph of the {\em right} $w$-action on $\frt_D$, i.e., $\frt_{w,D}$ consists of points $(t,w^{-1}t)$ (see Sec. \ref{ss:claction}). Let $\Heckep_{[w]}\subset\Heckep$ be the pre-image of $\frt_{w,D}$ under the map (\ref{eq:wpart}). They are disjoint over $\Mparrs$.

Consider $\Hecke^{\St}:=\Mpar\times_{\MHit\times X}\Mpar$ as a self-correspondence of $\Mpar$, where $\St$ stands for Steinberg. Then we have an embedding of correspondences $\Hecke^{\St}\subset\Heckep$ by identifying $\Hecke^{\St}$ with those Hecke modifications which does not change the underlying Hitchin pairs (i.e., only the Borel reduction is modified). Let $\Hecke^{\St}_w:=\Heckep_{[w]}\cap\Hecke^{\St}$.

The following lemma is immediate from definition.

\begin{lemma}\label{l:Wgraph}
Over $(\calA\times X)_0$, we identify $\calM^{\parab,0}$ with $\MHit\times_{\calA}\tcA^0$ by Lem. \ref{l:nuisom}. Then $\Hecke^{\St,0}_w$ is the graph of the right $w$-action on the second factor of $\MHit\times_{\calA}\tcA^0$.
\end{lemma}

\subsubsection{Group action on the Hecke correspondence} We define the {\em global affine Grassmannian} $\Grass_J$ of the group scheme $J$ over $\calA\times X$ as the functor which sends a scheme $S$ to the set of isomorphism classes of quadruples $(a,x,Q^J,\tau)$ where $(a,x)\in\calA(S)\times X(S)$, $Q^J$ is a $J_a$-torsor over $\XS$ and $\tau$ a trivialization of $Q^J$ over $\XS-\Gamma(x)$.

\begin{remark}
The fiber of $\Grass_J$ over $(a,x)\in(\calA\times X)(k)$ is canonically isomorphic to the local symmetry group $P_x(J_a)$ defined in Sec. \ref{ss:locglob}.
\end{remark}

Let $\tilGr_J$ be the pull back of $\Grass_J$ from $\calA\times X$ to $\tcA$.  Since $J$ is a commutative group scheme, the corresponding affine Grassmannians $\tilGr_J$ is naturally a group ind-scheme over $\tcA$. We have a homomorphism $\tilGr_J\to\tcP$ of sheaves of groups over $\tcA$ by forgetting the trivialization. Since $\tcP$ acts on $\Mpar$ over $\tcA$ by Lem. \ref{l:tcPaction}, we get an action of $\tilGr_J$ on $\Heckep$ by changing the second factor $(\calE_2,\varphi_2,\calE^B_{2,x})$ via the homomorphism $\tilGr_J\to\tcP$. This (left) action preserves the maps $\overleftarrow{h}$ and $\tilf\circ\overrightarrow{h}$. In particular, $\tilGr_J$ acts on $\Heckep_{[w]}$ for each $w\in W$.

\begin{lemma}\label{l:lattice}
Over $(\calA\times X)_0$, $\overleftarrow{h}^{0}_{[w]}:\Hecke^{\parab,0}_{[w]}\to\calM^{\parab,0}$ is a left $\tilGr_J^{0}$-torsor with a canonical trivialization given by the section $\calM^{\parab,0}\stackrel{\overleftarrow{h}^{\St,0}_w}{\cong}\Hecke^{\St,0}_w\hookrightarrow\Hecke^{\parab,0}_{[w]}$.
\end{lemma}
\begin{proof}
By Lem. \ref{l:Wgraph}, $\overleftarrow{h}^{\St,0}_w$ is an isomorphism and hence $\Hecke^{\St,0}_w$ gives a section of $\overleftarrow{h}^{0}_{[w]}$. To prove the lemma, we only need to show that for any point $m=(x,\calE_1,\varphi_1,\calE^B_{1,x})\in\Mpar(S)$ over $(a,x)\in(\calA\times X)_0(S)$, the fiber $(\overleftarrow{h}_{[w]})^{-1}(m)$ is a $P_x(J_a)$-torsor (although $P_x(J_a)$ was only defined for geometric points $(a,x)$ in Sec. \ref{ss:locglob}, the definition makes sense for any $S$-point $(a,x)$).

Let $\tilf(m)=(a,\tilx_1)$ for some $\tilx_1\in X_a$ over $x$. For a point $m'=(m,\calE_2,\varphi_2,\calE^B_{2,x},\alpha)\in(\overleftarrow{h}_{[w]})^{-1}(m)$, we have $\tilf(m)=(a,\tilx_2)$ where $\tilx_2=w^{-1}\cdot\tilx_1$ because $m'\in\Heckep_{[w]}$. Such a point $m'$ is completely determined by $(\calE_2,\varphi_2,\alpha)$ because the choice of the Borel reduction $\calE^B_{2,x}$ at $x$ is fixed by $\tilx_2$ (here we use the fact that $\delta_a(x)=0$). By analogy with Lem. \ref{l:hfiber}, we get a $P_x(J_a)$-equivariant isomorphism
\begin{equation*}
(\overleftarrow{h}_{[w]})^{-1}(m)\cong M^{\Hit}_x(\gamma)
\end{equation*}
for some $\gamma\in\frg(\hatO_x)$ with $\chi(\gamma)=a\in\frc(\hatO_x)$. Consider the regular locus $M^{\Hit,\reg}_x(\gamma)\subset M^{\Hit}_x(\gamma)$ (see \cite[3.3]{NgoFL}; again there it was defined for a geometric point $(a,x)$, but the definition makes sense for $S$-points), which is a torsor under $P_x(J_a)$. Therefore it suffices to show that $M^{\Hit,\reg}_x(\gamma)=M^{\Hit}_x(\gamma)$. Since $M^{\Hit,\reg}_x(\gamma)$ is open in $M^{\Hit}_x(\gamma)$, we only need to check that they are equal over every geometric point of $S$. Hence we reduce to the case where $S$ is the spectrum of an algebraically closed field. But in this case $\dim M^{\Hit}_x(\gamma)=\delta(a,x)=0$. By \cite[Corollaire 3.7.2 and Lemme 3.3.1]{NgoFL}, we conclude that $M^{\Hit}_x(\gamma)=M^{\Hit,\reg}_x(\gamma)$. This completes the proof of the lemma. 
\end{proof}

To describe $\tilGr_J$ over $\tcA^{\rs}$ more explicitly, we consider the global affine Grassmannian $\tilGr_T$ of the constant group scheme $T$ over $\tcA$. For any scheme $S$, $\tilGr_T(S)$ is the set of isomorphism classes of quadruples
$(a,\tilx,Q^T,\tau)$ where $(a,\tilx)\in\tcA(S)$, $Q^T$ is a $T$-torsor on $X_a$ and $\tau$ a trivialization of $Q^T$ on $X_a-\Gamma(\tilx)$.

The diagonal left action of $W$ on $T\times\tcA$ gives a $W$-action on $\tilGr_T$ over $\calA\times X$. On the other hand, $\tilGr_J$ also carries a left $W$-action induced from the left $W$-action on $\tcA$.

\begin{lemma}\label{l:JtoT}
There is a $W$-equivariant isomorphism of group ind-schemes over $\tcArs$:
\begin{equation*}
\jmath_{\Grass}:\tilGr_J^{\rs}\isom\tilGr_T^{\rs}.
\end{equation*}
\end{lemma}
\begin{proof}
For $(a,\tilx,Q^J,\tau)\in\tilGr_J(S)$ over $(a,\tilx)\in\tcArs(S)$, we get a $T$-torsor $Q^T$ over $X_a$ as in Construction \ref{cons:TtorsorP}. Since $a(x)\in\frc^{\rs}$, $q_a^{-1}(\Gamma(x))$ is a disjoint union
\begin{equation*}
q_a^{-1}(\Gamma(x))=\bigsqcup_{w\in W}\Gamma(w\tilx).
\end{equation*} 
The trivialization $\tau$ gives a trivialization $q_a^*\tau$ of $Q^T$ over $X_a-q_a^{-1}(\Gamma(x))$. We can glue the restriction of $Q^T$ to the open set $X_a-\bigsqcup_{w\neq e}\Gamma(w\tilx)$ with the trivial $T$-torsor over the open set $X_a-\Gamma(\tilx)$ via the trivialization $q_a^*\tau$. This gives a new $T$-torsor $Q^T_1$ over $X_a$ together with a tautological trivialization $\tau_1$ of $Q^T_1$ on $X_a-\Gamma(\tilx)$. We define the morphism $\jmath_{\Grass}$ by
\begin{equation*}
\jmath_{\Grass}(a,\tilx,Q^J,\tau)=(a,\tilx,Q^T_1,\tau_1)\in\tilGr_T(S).
\end{equation*}
The $W$-equivariance of $\jmath_{\Grass}$ follows from the $W$-equivariance of $\jmath$ in Lem. \ref{l:mapJtoT}. The fact that $\jmath_{\Grass}$ is an isomorphism follows from the fact that $\jmath$ is an isomorphism over $\frt^{\rs}$ (see Lem. \ref{l:mapJtoT}).
\end{proof}

It is well-known that the reduced structure of the Beilinson-Drinfeld Grassmannian $\Grass_{T}$ over a smooth curve is the constant group scheme $\xcoch(T)$. Therefore, the reduced structure of $\tilGr_T^{\rs}$ is the constant group scheme $\xcoch(T)$ over $\tcArs$. In other words, for each $\lambda\in\xcoch(T)$, we have a section $\tils_\lambda:\tcArs\to\tilGr_T^{\rs}$.

\begin{lemma}\label{l:extsection}
For each $\lambda\in\xcoch(T)$, the section $\tils_\lambda:\tcArs\to\tilGr_J^{\rs}$ extends to a section $\tils_\lambda:\tcA^0\to\tilGr_J^0$.
\end{lemma}
\begin{proof}
Let $Z_\lambda$ be the scheme-theoretic closure of the image $s_\lambda(\tcArs)$ in $\tilGr_J^0$. We only need to show that the projection induces an isomorphism $Z_\lambda\cong\tcA^0$. Let $Z'_\lambda=\calM^{\parab,0}\times_{\tcA^0}Z_\lambda$ be the closed substack of $\calM^{\parab,0}\times_{\tcA^0}\tilGr_J^0$. Since $\calM^{\parab,0}$ is faithfully flat over $\tcA^0$ by Cor. \ref{c:flat}, it suffices to show that $Z'_\lambda\cong\calM^{\parab,0}$ via the projection. Since $Z'_\lambda$ is flat over $Z_\lambda$, $Z'_\lambda$ is also the closure of $Z'^{\rs}_\lambda$. 

We first claim that $Z'_\lambda$ is proper over $\calM^{\parab,0}$. In fact, by Lem. \ref{l:lattice}, $\calH^0_{[e]}$ is isomorphic to the product $\calM^{\parab,0}\times_{\tcA^0}\tilGr_J^0$. Because $\calH_{[e]}$ is the inductive limit of proper substacks over $\Mpar$ by Lem. \ref{l:heckeft} and $Z'^{\rs}_\lambda$ (hence $Z'_\lambda$) is contained in one of these substacks, $Z'_\lambda$ is also proper over $\calM^{\parab,0}$.

We then claim that $Z'_\lambda$ is quasi-finite over $\calM^{\parab,0}$. In fact, its geometric fibers over $\calM^{\parab,0}$ are contained in $P_x(J_a)$ for $\delta(a,x)=0$, hence having dimension 0.

Now $Z'_\lambda$ is both proper and quasi-finite over $\calM^{\parab,0}$, it is therefore finite over $\calM^{\parab,0}$. Moreover, the projection $Z'_\lambda\to\calM^{\parab,0}$ is an isomorphism over the dense open subset $\Mparrs$. We conclude that $Z'_\lambda\cong\calM^{\parab,0}$ because $\calM^{\parab,0}$ is normal. This completes the proof.
\end{proof}

\begin{remark}\label{rm:slambda}
From Lem. \ref{l:extsection}, we see that each $\lambda\in\xcoch(T)$ gives a morphism
\begin{equation*}
s_\lambda:\tcA^{0}\xrightarrow{\tils_\lambda}\tilGr_{J}\to\calP
\end{equation*}
where the last arrow is the forgetful morphism (using the moduli meaning of $\tilGr_J$). Moreover, for any $w\in W$, we have
\begin{equation}\label{eq:swx}
s_{w\lambda}(\tila)=s_{\lambda}(w^{-1}\tila)
\end{equation}
for all $\tila\in\tcA$. In fact, this follows from the $W$-equivariance of the isomorphism $\jmath_{\Grass}$ in Lem. \ref{l:JtoT}.
\end{remark}

\begin{cor}\label{c:geomaction}
There exists a {\em right} action of $\tilW$ on $\calM^{\parab,0}$ over $(\calA\times X)_0$ such that the reduced structure of $\Heckeprs$ is the disjoint union of the graphs of this $\tilW$-action.
\end{cor}
\begin{proof}
We first define the right $\tilW$-action. Note that $\calM^{\parab,0}=\MHit\times_{\calA}\tcA^{0}$. For $(\lambda,w)\in\tilW$, we define its action on $(m,\tilx)\in\calM^{\parab,0}(S)=(\MHit\times_{\calA}\tcA^{0})(S)$ by
\begin{equation}\label{eq:rightaction}
(m,\tila)\cdot(\lambda,w):=(s_\lambda(\tila)m,w^{-1}\tila)
\end{equation}
where the action of $s$ on $m$ is given by the action of $\calP_a$ on $\MHit_a$ ($a$ is the image of $\tila$ in $\calA$). Using the relation (\ref{eq:swx}), it is easy to check that (\ref{eq:rightaction}) indeed gives a right action of $\tilW$: here we are using the fact that $\calP$ is commutative.

Next we verify that the reduced structure of $\Heckeprs$ is the disjoint union of the graphs of the right $\tilW$-action on $\calM^{\parab,0}$. By Lem. \ref{l:Wgraph}, $\Hecke^{\St,0}_w$ is the graph of the right $w$-action. By Lem. \ref{l:lattice}, the reduced structure of $\Heckeprs_{[w]}$ is the disjoint union of the $\xcoch(T)$-translations of $\Hecke^{\St,\rs}_w$. In other words, the reduced structure of $\Heckeprs_{[w]}$ is the disjoint union of the graphs of $(\lambda,w)$ for $\lambda\in\xcoch(T)$. This completes the proof.
\end{proof}

\begin{defn}\label{def:redHecke}
For each $\tilw\in\tilW$, the {\em reduced Hecke correspondence} $\calH_{\tilw}$ indexed by $\tilw$ is the closure (in $\Heckep$) of the graph of the right $\tilw$-action constructed in Cor. \ref{c:geomaction}.
\end{defn}

\begin{exam}
We describe the $\tilW$-action on $\calM^{\parab,0}$ in the case $G=\GL(n)$. We continue with the notation in Example \ref{ex:gl}. Notice that $(a,x)\in(\calA\times X)_0$ if and only if the spectral curve $Y_a$ is smooth at the points $p_a^{-1}(x)$. In this case, the parabolic Hitchin fiber $\Mpar_{a,x}$ consists of $\calF_0\in\overline{\stPic}(Y_a)$ and an ordering $(y_1,\cdots,y_n)$ of the $n$ points $p_a^{-1}(x)$ with multiplicities. For $w\in W=S_n$, its action on $\Mpar_{a,x}$ is the permutation action on the multi-set $p_a^{-1}(x)$, i.e., the change of the ordering. For $\lambda=(\lambda_1,\cdots,\lambda_n)\in\ZZ^n=\xcoch(T)$, its action on $\Mpar_{a,x}$ is given by tensoring $\calF_0$ with the line bundle $\calO_{Y_a}(\lambda_1y_1+\cdots+\lambda_ny_n)$, leaving the ordering unchanged.   
\end{exam}

% Main construction

\subsection{The affine Weyl group action on the parabolic Hitchin complex}\label{ss:main}

\begin{defn}\label{def:parcomplex}
The direct image complex $\fQl$ (resp. $\tfQl$) of the constant sheaf under the parabolic Hitchin fibration (resp. enhanced parabolic Hitchin fibration) is called the {\em parabolic Hitchin complex} (resp. the {\em enhanced parabolic Hitchin complex}).
\end{defn}

In this subsection, we prove the main result of this paper, namely, we construct an action of the extended affine Weyl group $\tilW$ on the parabolic Hitchin complex. This construction is the basis of all the subsequent development of the global Springer theory, and also justifies the title of the Thesis. We also construct a variant of this global Springer action in terms for the enhanced parabolic Hitchin complex. 

\subsubsection{The extended affine Weyl group action}
We apply the discussions in Sec. \ref{ss:grlike} to the situation where $S=\calA\times X$, $U=(\calA\times X)^{\rs}$, $X=\Mpar$ and the reduced Hecke correspondences $C=\calH_{\tilw}$ for each $\tilw\in\tilW$. Note that $\calH_{\tilw}$ is a graph-like correspondence: $\calH_{\tilw}|_U=\calH_{\tilw}^{\rs}$ is a graph (Cor. \ref{c:geomaction}), and $\calH_{\tilw}$ is the closure of $\calH^{\rs}_{\tilw}$. By the discussion in Sec. \ref{ss:grlike}, we get a map
\begin{equation}\label{eq:defaction}
[\calH_{\tilw}]_\#:\fQl\to\fQl.
\end{equation}

The main theorem of this paper is

\begin{theorem}\label{th:action} 
The assignment $\tilw\mapsto[\calH_{\tilw}]_\#$ for $\tilw\in\tilW$ gives a left action of $\tilW$ on $\fQl$.
\end{theorem}
\begin{proof}
By definition, the full Hecke correspondence $\Heckep$ has a natural convolution structure $\mu:\Heckep*\Heckep\to\Heckep$ by forgetting the middle $\Mpar$, which is obviously associative. Let $\Corr(\Heckep;\Ql,\Ql)$ be the direct limit
\begin{equation*}
\Corr(\Heckep;\Ql,\Ql):=\varinjlim_{V\subset\tilW}\Corr(\beta^{-1}(\Hecke^{\Bun}_V);\Ql,\Ql)
\end{equation*}
where $\Hecke^{\Bun}_V$ is the union of $\Hecke^{\Bun}_{\tilw}$ for $\tilw\in V$ and $V$ runs over all finite subsets of $\tilW$ such that $\Hecke^{\Bun}_V$ is closed in $\Hecke^{\Bun}$. Likewise we can define $\Corr(\Heckeprs;\Ql,\Ql)$ as a direct limit. The discussions of Sec. \ref{ss:convalg} can be applied to these direct limit situations, so that $\Corr(\Heckep;\Ql,\Ql)$ and $\Corr(\Heckeprs;\Ql,\Ql)$ have natural algebra structures given by convolutions. 

By Cor. \ref{c:geomaction}, $\calH^{\rs}$ is the disjoint union of graphs of the right $\tilW$-action on $\Mparrs$, therefore we have an algebra homomorphism
\begin{equation}\label{eq:homtilW}
\Ql[\tilW]\to\cohog{0}{\Heckeprs,\DD_{\overleftarrow{h}}}\cong \cohog{0}{\calH^{\rs},\DD_{\overleftarrow{h}}}=\Corr(\calH^{\rs};\Ql,\Ql).
\end{equation}
which sends $\tilw$ to $[\calH^{\rs}_{\tilw}]$. Here $\Ql[\tilW]$ is the group algebra of $\tilW$.

By Lem. \ref{l:HeckeG2} below, any finite type substack of $\Heckep$ satisfies the condition (G-2) in Def. \ref{def:grlike} with respect to $(\calA\times X)^{\rs}\subset\calA\times X$. Therefore, by Prop. \ref{p:alghom}, the action of $\Corr(\Heckep;\Ql,\Ql)$ on $\fQl$ factors through $\Corr(\Heckeprs;\Ql,\Ql)$; by the homomorphism (\ref{eq:homtilW}), we get an action of $\tilW$ on $\fQl$ sending $\tilw$ to $[\calH_{\tilw}]_\#$. The theorem is proved.
\end{proof}

\begin{lemma}\label{l:HeckeG2}
Any finite type substack of $\Heckep$ satisfies the condition (G-2) in Def. \ref{def:grlike} with respect to $(\calA\times X)^{\rs}\subset\calA\times X$.
\end{lemma}
\begin{proof}
Fix an integer $\delta>0$, let $\Mpar_\delta\subset\Mpar$ and $\Heckep_\delta\subset\Heckep$ be the pre-images of $(\calA\times
X)_{\delta}$. The relative dimension of $\overleftarrow{h}:\Heckep_{\delta}\to\Mpar_{\delta}$ is less or equal to $\delta$ because the fibers of $\overleftarrow{h}$ are affine Springer fibers of dimension $\delta$ by Lem. \ref{l:hfiber}. By Cor.
\ref{c:localdel}, $\codim_{\calA\times X}((\calA\times X)_{\delta})\geq\delta+1$. Since $\fpar$ is flat, we have $\codim_{\Mpar}(\Mpar_\delta)\geq\delta+1$. Therefore we conclude $\dim\Heckep_{\delta}\leq\dim\Mpar-\delta-1+\delta=\dim\Mpar-1$.

For $\delta=0$, let $V=(\calA\times X)_0-(\calA\times X)^{\rs}$ and $\Mpar_V\subset\Mpar, \Heckep_V\subset\Heckep$ be the pre-images. Then obviously $\codim_{\calA\times X}(V)=\codim_{\Mpar}(\Mpar_V)\geq1$. Since the fibers of $\overleftarrow{h}$ are zero-dimensional over $V$, we can still conclude that $\dim\Heckep_V\leq\dim\Mpar-1$. 
\end{proof}

\subsubsection{The enhanced action} For each $\tilw\in\tilW$ with image $w\in W$ under the projection $\tilW\to W$, the correspondence $\calH_{\tilw}$ can be viewed as a correspondence over $\tcA$, which we denote by $\calHn_{\tilw}$:
\begin{equation*}
\xymatrix{ & \calHn_{\tilw} \ar[dl]\ar[dr] & \\
 \Mpar\ar[dr]_{R_w\circ\tilf} & & \Mpar\ar[dl]^{\tilf}\\
& \tcA & }
\end{equation*}
Here we have composed the left hand side $\tilf$ with the right action of $w$ on $\tcA$ (denoted by $R_w$). Therefore $[\calHn_{\tilw}]_\#$ induces an isomorphism
\begin{equation}\label{eq:enhancefclass}
[\calHn_{\tilw}]_\#:\tfQl\to R_{w,*}\tfQl.
\end{equation}

\begin{prop}\label{p:enhancedaction} The maps $[\calHn_{\tilw}]_\#$ in (\ref{eq:enhancefclass}) give a $\tilW$-equivariant structure on $\tfQl$, compatible with the right $\tilW$-action on $\tcA$ through the quotient $\tilW\to W$.
\end{prop}
\begin{proof}
To check that (\ref{eq:enhancefclass}) does give an equivariant structure, we pick two elements $\tilw_1,\tilw_2\in\tilW$ with projections $w_1,w_2\in W$. We have to show that the following diagram is commutative
\begin{equation}\label{eq:enhancecomm}
\xymatrix{\tfQl\ar[rr]^{[\calHn_{\tilw_1\tilw_2}]_\#}\ar[dr]_{[\calHn_{\tilw_2}]_\#} & & R_{w_2,*}R_{w_1,*}\tfQl\ar@{=}[r] & R_{w_1w_2,*}\tfQl\\
& R_{w_2,*}\tfQl\ar[ur]_{R_{w_2,*}[\calHn_{\tilw_1}]_\#}} 
\end{equation}
Let $\calH^{\natural,w_2}_{\tilw_1}$ be the same correspondence as $\calHn_{\tilw_1}$ between $\Mpar$ and itself, but the structure morphisms of the left and right $\Mpar$ to $\tcA$ are of the form $R_{w_1w_2}\tilf:\Mpar\to\tcA$ and $R_{w_2}\tilf:\Mpar\to\tcA$. Then $\calH^{\natural,w_2}_{\tilw_1}$ and $\calHn_{\tilw_2}$ are composable. By Lem. \ref{l:HeckeG2}, the correspondences $\calH^{\natural,w_2}_{\tilw_1},\calHn_{\tilw_2}$ and their composition are all graph-like with respect to $\tcArs\subset\tcA$. By Prop. \ref{p:comp}, we have
\begin{equation}\label{eq:applycomp}
[\calH^{\natural,w_2}_{\tilw_1}]_\#\circ[\calHn_{\tilw_2}]_\#=[\calH^{\natural,w_2}_{\tilw_1}*\calHn_{\tilw_2}]_\#.
\end{equation}
It is easy to verify that $\calHn_{\tilw_1\tilw_2}$ and $\calH^{\natural,w_2}_{\tilw_1}*\calHn_{\tilw_2}$ are isomorphic over $\tcArs$, because they are the graph of the same automorphism of $\Mparrs$. Since they both satisfy the condition (G-2) with respect to $\tcArs\subset\tcA$, Lem. \ref{l:openpart} implies
\begin{equation}\label{eq:compandH}
[\calH^{\natural,w_2}_{\tilw_1}*\calHn_{\tilw_2}]_\#=[\calHn_{\tilw_1\tilw_2}]_\#.
\end{equation}
Combining (\ref{eq:compandH}) and (\ref{eq:applycomp}), we get
\begin{equation}\label{eq:almostimply}
[\calH^{\natural,w_2}_{\tilw_1}]_\#\circ[\calHn_{\tilw_2}]_\#=[\calHn_{\tilw_1\tilw_2}]_\#.
\end{equation}
Notice that
\begin{equation*}
[\calH^{\natural,w_2}_{\tilw_1}]_\#=R_{w_2,*}([\calHn_{\tilw_1}]_\#):R_{w_2,*}\tfQl\to R_{w_2,*}R_{w_1,*}\tfQl,
\end{equation*}
therefore (\ref{eq:almostimply}) implies the commutativity of the diagram (\ref{eq:enhancecomm}). This completes the proof.
\end{proof}

% Commuting with Verdier duality

\subsection{The affine Weyl group action and Verdier duality}
In this subsection, we check that the $\tilW$-action on $\fQl$ essentially commutes with Verdier duality. 

Let $d=\dim\Mpar$. We fix a fundamental class of $\Mpar$, hence fixing an isomorphism \begin{equation}\label{eq:funclM}
u=[\Mpar]:\const{\Mpar}[d](d/2)\cong\DD_{\Mpar}[d](d/2).
\end{equation}
This induces an isomorphism
\begin{equation*}
v:\fQl[d](d/2)\isom\fpar_*\DD_{\Mpar}[d](d/2)=\DD(\fQl[d](d/2)).
\end{equation*}

\begin{prop}\label{p:VDWaff}
For any element $\tilw\in\tilW$, there is a commutative diagram in $D^b_c(\calA\times X)$:
\begin{equation}\label{d:VDcomm}
\xymatrix{\fQl[d](d/2)\ar[rrr]^{\tilw}\ar[d]^{v} &&& \fQl[d](d/2)\ar[d]^{v}\\
\DD(\fQl[d](d/2))\ar[rrr]^{\DD(\tilw^{-1})} &&& \DD(\fQl[d](d/2))}
\end{equation}
where the horizontal maps come from the $\tilW$-action constructed in Th. \ref{th:action}.

Similar result holds for the $\xcoch(T)$-action on $\tfQl[d](d/2)$.
\end{prop}
\begin{proof}
Let $L:=\const{\Mpar}[d](d/2)$. By Lem. \ref{l:VDcorr}, the map $\DD(\tilw^{-1})$ in diagram (\ref{d:VDcomm}) is given by
\begin{equation*}
\DD([\calH_{\tilw^{-1}}]_\#)=\DD([\calH_{\tilw^{-1}}])_\#
\end{equation*}
where $\DD([\calH_{\tilw^{-1}}])$ is the Verdier dual of the cohomological correspondence $[\calH_{\tilw^{-1}}]\in\Corr(\calH_{\tilw^{-1}};L,L)$.
It is clear that $\calH_{\tilw^{-1}}$ and $\calH\dual_w$ coincide over $(\calA\times X)^{\rs}$ since they are both the graph of the right $\tilw^{-1}$-action on $\Mpar$. Therefore, taking closures in $\Heckep$, we get $\calH_{\tilw^{-1}}=\calH\dual_{\tilw}$ as self-correspondences of $\Mpar$ over $\calA\times X$. To prove the proposition, we only have to show that under the following two maps
\begin{equation*}
\xymatrix{[\calH_{\tilw}]\in\Corr(\calH_{\tilw};L,L)\ar[rr]^{\Corr(\calH_{\tilw};u,u)} && \Corr(\calH_{\tilw};\DD L,\DD L)\\
[\calH_{\tilw^{-1}}]\in\Corr(\calH_{\tilw^{-1}};L,L)=\Corr(\calH\dual_{\tilw};L,L)\ar[urr]^{\DD}}
\end{equation*}
the elements $[\calH_{\tilw}]$ and $[\calH_{\tilw^{-1}}]$ have the same image in $\Corr(\calH_{\tilw};\DD L,\DD L)$. Here the map $u:L\isom\DD L$ is defined in (\ref{eq:funclM}). Since the correspondences involved are all graph-like, by Lem. \ref{l:openpart}, it suffices to check the coincidence of the two images in $\Corr(\calH^{\rs}_{\tilw};\DD L,\DD L)$, which is obvious. 
\end{proof}

%  Appendix

\appendix

\section{Generalities on cohomological correspondences}\label{s:corr}

In this appendix, we review the formalism of cohomological correspondences. The results in  Sec. \ref{ss:corr} through Sec. \ref{ss:VDcorr} should be standard and we partially follow the presentation of \cite{SGA5}. We introduce a nice class of correspondences called {\em graph-like correspondences} in Sec. \ref{ss:grlike}, which will be used in the construction of the affine Weyl group action on the parabolic Hitchin complex.

We sometimes put a label over an arrow to describe the nature of the map. The label ``\bch'' means proper base change; ``\adj'' means adjunction; ``$\stsh$'' means the natural transformation $\phi^*f^!\to f^!_1\psi^*$ (adjoint to the proper base change) associated to the following Cartesian diagram
\begin{equation*}
\xymatrix{X_1\ar[r]^{\phi}\ar[d]^{f_1} & X\ar[d]^{f}\\Y_1\ar[r]^{\psi} & Y}
\end{equation*}

\subsection{Cohomological correspondences}\label{ss:corr}
We recall the general formalism of cohomological correspondences, following \cite{SGA5}. Consider the following correspondence diagram
\begin{equation}\label{d:corr}
\xymatrix{ & C\ar[dl]_{\overleftarrow{c}}\ar[dr]^{\overrightarrow{c}} & \\
 X\ar[dr]_{f} & & Y\ar[dl]^{g}\\
& S & }
\end{equation}
where all the spaces are Deligne-Mumford stacks and all maps are of finite type. {\em We always assume that $\overleftrightarrow{c}=(\overleftarrow{c},\overrightarrow{c}):C\to X\times_S Y$ is proper}. Let $\overleftarrow{p},\overrightarrow{p}$ be the projections from $X\times_SY$ to $X$ and $Y$ respectively.

\begin{defn}\label{def:cohocorr} A {\em cohomological correspondence between a complex $\calF\in D^b(X)$ and a complex $\calG\in D^b(Y)$ with support on $C$} is an element in
\begin{equation*}
\Corr(C;\calF,\calG):=\Hom_{C}(\overrightarrow{c}^*\calG,\overleftarrow{c}^!\calF).
\end{equation*} 
\end{defn}

A such cohomological correspondence $\zeta$ induces a morphism 
\begin{equation}\label{eq:defeff}
\zeta_\#:g_!\calG\to f_*\calF
\end{equation}
which is defined by the following procedure:
\begin{equation}\label{eq:defcorr}
(-)_\#:\Corr(C;\calF,\calG)\xrightarrow{\alpha}\Corr(X\times_SY;\calF,\calG)\xrightarrow{\beta}\Hom_{S}(g_!\calG,f^*\calF). 
\end{equation}
where $\alpha$ is the composition
\begin{eqnarray*}
\Corr(C;\calF,\calG)&\xrightarrow{\overleftrightarrow{c}_*}&\Hom_{X\times_SY}(\overleftrightarrow{c}_*\overrightarrow{c}^*\calG,\overleftrightarrow{c}_*\overleftarrow{c}^!\calF)\\
&=&\Hom_{X\times_SY}(\overleftrightarrow{c}_*\overleftrightarrow{c}^*\overrightarrow{p}^*\calG,\overleftrightarrow{c}_!\overleftrightarrow{c}^!\overleftarrow{p}^!\calF)\textup{ (}\overleftrightarrow{c}\textup{ is proper)}\\
&\xrightarrow{\adj}&\Corr(X\times_SY;\calF,\calG)
\end{eqnarray*}
and $\beta$ is the composition
\begin{equation}\label{eq:defnsharp}
\Corr(X\times_SY;\calF,\calG)\stackrel{\adj}{=}\Hom_S(\calG,\overrightarrow{p}_*\overleftarrow{p}^!\calF)\stackrel{\bch}{\cong}\Hom_{S}(\calG,g^!f_*\calF)\stackrel{\adj}{=}\Hom_S(g_!\calG,f_*\calF).
\end{equation}

The morphism $\alpha$ is a special case of the {\em push-forward} of cohomological correspondences: Suppose $\gamma:C\to C'$ is a proper map of correspondences between $X$ and $Y$ over $S$, then we can define
\begin{equation*}
\gamma_*:\Corr(C;\calF,\calG)\to\Corr(C';\calF,\calG). 
\end{equation*}
in the same way as we defined $\alpha$. It follows directly from the definition that

\begin{lemma}\label{l:pushclass}
For any $\zeta\in\Corr(C;\calF,\calG)$, we have
\begin{equation*}
(\gamma_*\zeta)_\#=\zeta_\#\in\Hom_S(g_!\calG,f_*\calF).
\end{equation*}
\end{lemma}

We will mainly be interested in the special case where $\calF$ and $\calG$ are the constant sheaves (in degree 0) on $X$ and $Y$. In this case, we have
\begin{equation*}
\Corr(C;\const{X},\const{Y})=\Hom_{C}(\overrightarrow{c}^*\const{Y},\overleftarrow{c}^!\const{X})=H^0(C,\DD_{\overleftarrow{c}}).
\end{equation*}
Here $\DD_{\overleftarrow{c}}$ means the dualizing complex relative to the morphism $\overleftarrow{c}$.

% pull-back

\subsection{Pull-back of correspondences}\label{ss:pullback} 
Let $C$ be a correspondence between $X$ and $Y$ over $S$. Let $\beta:S'\to S$ be a morphism. We base change the whole situation of diagram (\ref{d:corr}) to $S'$ and get a correspondence
\begin{equation*}
\xymatrix{ & C'\ar[dl]_{\overleftarrow{c'}}\ar[dr]^{\overrightarrow{c'}} & \\
X'\ar[dr]_{f'} & & Y'\ar[dl]^{g'}\\
& S' & }
\end{equation*}

Let $\gamma:C'\to C, \phi:X'\to X, \psi:Y'\to Y, \theta=(\phi,\psi):X'\times_{S'}Y'\to X\times_SY$ be the morphisms base-changed from $\beta$. We define the {\em pull-back} map of cohomological correspondences:
\begin{equation*}
\gamma^*: \Corr(C;\calF,\calG)\to\Corr(C';\phi^*\calF,\psi^*\calG)
\end{equation*}
as follows: for $\zeta\in\Corr(C;\calF,\calG)$, $\gamma^*\zeta$ is defined as the composition
\begin{equation}\label{eq:pullcorr}
\overrightarrow{c'}^*\psi^*\calG=\gamma^*\overrightarrow{c}^*\calG\xrightarrow{\gamma^*\zeta}\gamma^*\overleftarrow{c}^!\calF\xrightarrow{\stsh}\overleftarrow{c'}^!\phi^*\calF.
\end{equation}

\begin{lemma}\label{l:pullbase}
Suppose $f:X\to S$ is proper, then for any $\zeta\in\Corr(C;\calF,\calG)$ we have a commutative diagram
\begin{equation}\label{d:corrpull}
\xymatrix{\Corr(C;\calF,\calG)\ar[r]^{(-)_\#}\ar[d]^{\gamma^*} & \Hom_S(g_!\calG,f_*\calF)\ar[d]^{\beta^*}\\
\Corr(C';\phi^*\calF,\psi^*\calG)\ar[r]^{(-)_\#} & \Hom_{S'}(\beta^*g_!\calG,\beta^*f_*\calF)\ar@{=}[r]^{\bch}& \Hom_{S'}(g'_!\psi^*\calG,f'_*\phi^*\calF)}.
\end{equation}
\end{lemma}
\begin{proof}
The proof is a diagram chasing. We have a commutative diagram of functors
\begin{equation}\label{d:basecd}
\xymatrix{\psi^*\overrightarrow{c}_*\overleftarrow{c}^!\ar[r]\ar@{=}[d]^{\bch} & \psi^*\overrightarrow{p}_*\overleftarrow{p}^!\ar@{=}[r]^{\bch}\ar@{=}[d]^{\bch} & \psi^*g^!f_*\ar[d]^{\stsh}\\
\overrightarrow{c'}_*\gamma^*\overleftarrow{c}^!\ar[r]\ar[d]^{\stsh} &\overrightarrow{p'}_*\theta^*\overleftarrow{p}^!\ar[d]^{\stsh} & g'^!\beta^*f_*\ar@{=}[d]^{\bch}\\
\overrightarrow{c'}_*\overleftarrow{c'}^!\phi^*\ar[r] & \overrightarrow{p'}_*\overleftarrow{p'}^!\phi^*\ar@{=}[r]^{\bch} & g'^!f'_*\phi^*}
\end{equation}
here the unlabeled maps are defined in a similar way as the map $\alpha$ in (\ref{eq:defcorr}).

Therefore the following diagram is also commutative
\begin{equation*}
\xymatrix{\psi^*\calG\ar[r]^{\psi^*g^!\zeta_\#}\ar[dr]_{g'^!(\gamma^*\zeta)_\#} & \psi^*g^!f_*\calF\ar[d]\ar[r]^{\stsh} & g'^!\beta^*f_*\calF\\
& g'^!f'_*\phi^*\calF\ar@{=}[ur]_{\bch} &}
\end{equation*}
where the commutativity of the left triangle follows from diagram (\ref{d:basecd}) applied to $\calF$. Applying the adjunction of $(g'_!,g'^!)$ to the above diagram gives the commutative diagram
\begin{equation*}
\xymatrix{\beta^*g_!\calG\ar[r]^{\beta^*\zeta_\#}\ar[d]^{\bch}_{\cong} & \beta^*f_*\calF\ar[d]^{\bch}_{\cong}\\
g'_!\psi^*\calG\ar[r]^{(\gamma^*\zeta)_\#} & f'_*\phi^*\calF}
\end{equation*} 
which is equivalent to the diagram (\ref{d:corrpull}).
\end{proof}

% compositions

\subsection{Composition of correspondences}\label{ss:compcorr}
Let $C_1$ be a correspondence between $X$ and $Y$ over $S$, and $C_2$ be a correspondence between $Y$ and $Z$ over $S$. {\em Assume that $Y$ is proper over $S$}. The {\em composition} $C=C_1*C_2$ of $C_1$ and $C_2$ is defined to be $C_1\times_Y C_2$, viewed as a correspondence between $X$ and $Z$ over $S$:
\begin{equation}\label{d:comp}
\xymatrix{& &
C\ar[dl]^{\overleftarrow{d}}\ar@/_2pc/[ddll]_{\overleftarrow{c}}\ar[dr]_{\overrightarrow{d}}\ar@/^2pc/[drdr]^{\overrightarrow{c}} & & \\
& C_1\ar[dl]^{\overleftarrow{c_1}}\ar[dr]^{\overrightarrow{c_1}} & & C_2\ar[dl]_{\overleftarrow{c_2}}\ar[dr]_{\overrightarrow{c_2}} &\\
X\ar[drr]_{f} & & Y\ar[d]^{g} & & Z\ar[dll]^{h}\\
& & S & &}
\end{equation}
Note that the properness of $Y$ over $S$ ensures the properness of $C$ over $X\times_S Z$.

Let $\calF,\calG,\calH$ be complexes on $X,Y$ and $Z$ respectively. We define the {\em convolution product}
\begin{equation}\label{eq:convcorr}
\circ_Y:\Corr(C_1;\calF,\calG)\otimes\Corr(C_2;\calG,\calH)\to\Corr(C;\calF,\calH)
\end{equation}
as follows. Let $\zeta_1\in\Corr(C_1;\calF,\calG)$ and $\zeta_2\in\Corr(C_2;\calG,\calH)$, then $\zeta_1\circ_Y\zeta_2$ is given by
\begin{equation*}
\overrightarrow{c}^*\calH=\overrightarrow{d}^*\overrightarrow{c_2}^*\calH\xrightarrow{\zeta_2}\overrightarrow{d}^*\overleftarrow{c_2}^!\calG\xrightarrow{\stsh}\overleftarrow{d}^!\overrightarrow{c_1}^*\calG\xrightarrow{\zeta_1}\overleftarrow{d}^!\overleftarrow{c_1}^!\calF=\overleftarrow{c}^!\calF.
\end{equation*}

\begin{lemma}\label{l:conv}
For $\zeta_1\in\Corr(C_1;\calF,\calG)$, $\zeta_2\in\Corr(C_2;\calG,\calH)$, we have
\begin{equation*}
\zeta_{1,\#}\circ\zeta_{2,\#}=(\zeta_1\circ_Y\zeta_2)_{\#}:h_!\calH\to f_*\calF.
\end{equation*}
\end{lemma}
\begin{proof}
The proof is again a diagram chasing:
\begin{equation*}
\xymatrix{& \overrightarrow{c_2}_*\overleftarrow{c_2}^!\calG\ar@/^1pc/[rr]^{\zeta_1}\ar[r]_(.4){\adj}\ar[d] & \overrightarrow{c_2}_*\overleftarrow{c_2}^!\overrightarrow{c_1}_*\overrightarrow{c_1}^*\calG\ar[r]\ar[d] & \overrightarrow{c_2}_*\overleftarrow{c_2}^!\overrightarrow{c_1}_*\overleftarrow{c_1}^!\calF\ar@{=}[r]^{\bch}\ar[d] & \overrightarrow{c_2}_*\overrightarrow{d}_*\overleftarrow{d}^!\overleftarrow{c_1}^!\calF\ar[d]\\
\calH\ar[ur]^{\zeta_2}\ar[r]_{\zeta_{2,\#}} & h^!g_*\calG\ar[r]^(.4){\adj}\ar@/_1pc/[rr]_{\zeta_{1,\#}} & h^!g_*\overrightarrow{c_1}_*\overrightarrow{c_1}^*\calG\ar[r] & h^!g_*g^!f_*\calF\ar[r]^{g_!=g_*,\adj} & h^!f_*\calF}
\qedhere
\end{equation*}
\end{proof}

Consider the correspondences $C_i$ between $X_i$ and $X_{i+1}$, $i=1,2,3$. Assume $X_2,X_3$ are proper over $S$. It follows from the definition of convolution that:
\begin{lemma}\label{l:ass}
The convolution product is associative. More precisely, for $\calF_i\in D^b(X_i),i=1,\cdots,4$ and $\zeta_i\in\Corr(C_i;\calF_i,\calF_{i+1})$, $i=1,2,3$, we have
\begin{equation*}
(\zeta_1\circ_{X_2}\zeta_2)\circ_{X_3}\zeta_3=\zeta_1\circ_{X_2}(\zeta_2\circ_{X_3}\zeta_3).
\end{equation*}
\end{lemma}

% Verdier duality

\subsection{Verdier duality and correspondences}\label{ss:VDcorr}
In this subsection, we study the interaction between Verdier duality and cohomological correspondences. We continue to use the notation from Sec. \ref{ss:corr}. Interchanging $X$ and $Y$ in the diagram (\ref{d:corr}), the same stack $C$ can be viewed as a correspondence between $Y$ and $X$. We denote this correspondence by $C\dual$, and call it the {\em transposition} of $C$. The transposition $C\dual$ has two projections:
\begin{eqnarray*}
\overleftarrow{c\dual}=\overrightarrow{c}:C\to Y;\hspace{1cm}\overrightarrow{c\dual}=\overleftarrow{c}:C\to X.
\end{eqnarray*}

The Verdier duality functor gives an isomorphism
\begin{equation*}
\DD(-):\Corr(C;\calF,\calG)\isom\Corr(C\dual;\DD\calG,\DD\calF)
\end{equation*}
which sends the map $\zeta:\overrightarrow{c}^*\calG\to\overleftarrow{c}^!\calF$ to its Verdier dual
\begin{equation*}
\DD\zeta:\overrightarrow{c\dual}^*\DD\calF=\DD(\overleftarrow{c}^!\calF)\to\DD(\overrightarrow{c}^*\calG)=\overleftarrow{c\dual}^!\DD\calG.
\end{equation*}

On the other hand, the Verdier duality functor also gives an isomorphism
\begin{equation*}
\DD(-):\Hom_S(g_!\calG,f_*\calF)\isom\Hom_{S}(f_!\DD\calF,g_*\DD\calG).
\end{equation*}

\begin{lemma}\label{l:VDcorr}
For any $\zeta\in\Corr(C;\calF,\calG)$, we have
\begin{equation*}
\DD(\zeta_\#)=(\DD\zeta)_\#\in\Hom_{S}(f_!\DD\calF,g_*\DD\calG).
\end{equation*}
\end{lemma}
\begin{proof}
First, we claim that the following diagram is commutative
\begin{equation*}
\xymatrix{\Corr(C;\calF,\calG)\ar[r]^{\overleftrightarrow{c}_*}\ar[d]^{\DD} & \Corr(X\times_SY;\calF,\calG)\ar[d]^{\DD}\\
\Corr(C\dual;\DD\calG,\DD\calF)\ar[r]^{\overleftrightarrow{c}_*} & \Corr(Y\times_SX;\DD\calG,\DD\calF)}
\end{equation*}
In fact, this follows from the fact that $\DD$ transforms the adjunction $\id\to\overleftrightarrow{c}_*\overleftrightarrow{c}^*$ to the adjunction $\overleftrightarrow{c}_!\overleftrightarrow{c}^!\to\id$. Therefore we have reduced the situation to the case $C=X\times_SY$ (hence $C\dual=Y\times_SX$).

For $\zeta\in\Corr(X\times_SY;\calF,\calG)=\Hom_{X\times_SY}(\overrightarrow{p}^*\calG,\overleftarrow{p}^!\calF)$, the map $\zeta_\#$ is given by (see (\ref{eq:defnsharp})):
\begin{equation*}
g_!\calG\xrightarrow{\adj}g_!\overrightarrow{p}_*\overrightarrow{p}^*\calG\xrightarrow{g_!\overrightarrow{p}_*(\zeta)}g_!\overrightarrow{p}_*\overleftarrow{p}^!\calF\xrightarrow{\bch} g_!g^!f_*\calF\xrightarrow{\adj}f_*\calF.
\end{equation*}
Dualizing the above maps, we get
\begin{equation}\label{eq:midd}
f_!\DD\calF\xrightarrow{\adj}g_*g^*f_!\DD\calF\xrightarrow{\bch}g_*\overrightarrow{p}_!\overleftarrow{p}^*\DD\calF\xrightarrow{g_!\overrightarrow{p}_!(\DD\zeta)}g_*\overrightarrow{p}_!\overrightarrow{p}^!\DD\calG\xrightarrow{\adj}g_*\DD\calG.
\end{equation}

We have to show that (\ref{eq:midd}) coincides with $(\DD\zeta)_\#$. In view of the definition of the assignment $\zeta\mapsto\zeta_\#$ in (\ref{eq:defnsharp}), we have to check the commutativity of the following diagram 
\begin{equation*}
\xymatrix{f_!\DD\calF\ar[d]^{\adj}\ar[r]^{\adj}\ar@{}[dr]|{(a)} & f_!\overleftarrow{p}_*\overleftarrow{p}^*\DD\calF\ar[rr]^{f_!\overleftarrow{p}_*(\DD\zeta)}\ar[d]^{!*\to*!}\ar@{}[drr]|{(b)} && f_!\overleftarrow{p}_*\overrightarrow{p}^!\DD\calG\ar[d]^{!*\to*!}\ar[r]^{\bch}\ar@{}[dr]|{(c)} & f_!f^!g_*\DD\calG\ar[d]^{\adj}\\
g_*g^*f_!\DD\calF\ar[r]^{\bch} & g_*\overrightarrow{p}_!\overleftarrow{p}^*\DD\calF\ar[rr]^{g_*\overrightarrow{p}_!(\DD\zeta)} && g_*\overrightarrow{p}_!\overrightarrow{p}^!\DD\calG\ar[r]^{\adj} & g_*\DD\calG}
\end{equation*}
Here the first row is the map (\ref{eq:defnsharp}) and the second row is (\ref{eq:midd}). The two vertical arrows labeled by $!*\to*!$ are given by the natural transformation $f_!\overleftarrow{p}_*\to g_*\overrightarrow{p}_!$, hence the square (b) is commutative. It remains to show that the squares (a) and (c) are commutative. Applying the adjunction $(g^*,g_*)$ to the square (a), it becomes
\begin{equation*}
\xymatrix{g^*f_!\DD\calF\ar[rr]^{\adj}\ar[d]^{\bch} & & g ^*f_!\overleftarrow{p}_*\overleftarrow{p}^*\DD\calF\ar[d]^{\bch}\\
\overrightarrow{p}_!\overleftarrow{p}^*\DD\calF\ar[rr]^{\overrightarrow{p}_!\overleftarrow{p}^*(\adj)}\ar@/_1.5pc/[rrrr]_{\id} & & \overrightarrow{p}_!\overleftarrow{p}^*\overleftarrow{p}_*\overleftarrow{p}^*\DD\calF\ar[rr]^{\overrightarrow{p}_!(\adj)} &&\overrightarrow{p}_!\overleftarrow{p}^*\DD\calF}
\end{equation*}
Now the upper square is commutative by the functoriality of the proper base change isomorphism; the lower row is commutative because the composition $\overleftarrow{p}^*\xrightarrow{\overleftarrow{p}^*(\adj)}\overleftarrow{p}^*\overleftarrow{p}_*\overleftarrow{p}^*\xrightarrow{\adj(\overleftarrow{p}^*)}\overleftarrow{p}^*$ is identity. The commutativity of the square (c) can be verified in a similar way. This completes the proof. 
\end{proof}

% Graph-like

\subsection{Integration along a graph-like correspondence}\label{ss:grlike} 
In this subsection, we assume $X$ to be smooth and equidimensional of dimension $d$. We will introduce a class of correspondences which will be useful in construting global Springer actions.

\begin{defn}\label{def:grlike} Let $U\subset S$ be an open subscheme. A correspondence $C$ between $X$ and $Y$ over $S$ is said to be {\em left graph-like with respect to $U$} if it satisfies the following conditions:
\begin{enumerate}
\item[(G-1)] {\em The projection $\overleftarrow{c}:C_U\to X_U$ is \'{e}tale.}
\item[(G-2)] {\em $\dim C_U\leq d$ and the image of $C-C_U\to X\times_{S}Y$ has
dimension strictly less than $d$.}
\end{enumerate}
Similarly, $C$ is said to be {\em right graph-like with respect to $U$} if $\overrightarrow{c}:C_U\to X_U$ is \'{e}tale and (G-2) is satisfied; $C$ is said to be {\em graph-like with respect to $U$} if it is both left and right graph-like.
\end{defn}
Note that the inequality $\dim C_U\leq d$ is certainly implied by (G-1); we leave it here because sometimes we will refer to condition (G-2) alone without assuming (G-1).

\begin{lemma}\label{l:openpart}
Suppose $C$ is a correspondence between $X$ and $Y$ over $S$ satisfying (G-2) with respect to $U\subset S$. Let $\zeta,\zeta'\in
\Corr(C;\const{X},\const{Y})$. If $\zeta|_U=\zeta'|_U\in\Corr(C_U;\const{X_U},\const{Y_U})$, then $\zeta_\#=\zeta'_\#\in\Hom_S(g_!\const{Y},f_*\const{X})$.
\end{lemma}
\begin{proof}
Let $Z$ be the image of $\overleftrightarrow{c}$ and $\overleftarrow{z},\overrightarrow{z}:Z\to X$ be the projections. Under the above assumptions, after choosing a fundamental class of $X$, we can identify $\const{X}$ with $\DD_X[-2d](-d)$, hence identify $\DD_{\overleftarrow{z}}$ with $\DD_Z[-2d](-d)$. Similar remark applies to $Z_U$. Consider the restriction map
\begin{equation*}
j^*:\cohog{-2d}{Z,\DD_Z}(-d)=\hBM{2d}{Z}(-d)\to\hBM{2d}{Z_U}(-d)=\cohog{-2d}{Z_U,\DD_{Z_U}}(-d).
\end{equation*}
Both sides have a basis consisting of fundamental classes of $d$-dimensional irreducible components of $Z$ or $Z_U$. By condition (G-2), the $d$-dimensional irreducible components of $Z$ and $Z_U$ are naturally in bijection. Therefore $j^*$ is an isomorphism, and the restriction map
\begin{equation*}
\Corr(Z;\const{X},\const{Y})=\cohog{0}{Z,\DD_{\overleftarrow{z}}}\to \cohog{0}{Z_U,\DD_{\overleftarrow{z}}}=\Corr(Z_U;\const{X_U},\const{Y_U})
\end{equation*}
is also an isomorphism. Since $\zeta|_U=\zeta'|_U\in\Corr(C_U;\const{X_U},\const{Y_U})$, hence
$(\overleftrightarrow{c}_*\zeta)|_U=(\overleftrightarrow{c}_*\zeta')|_U\in
\Corr(Z_U;\const{X_U},\const{Y_U})$, therefore
$\overleftrightarrow{c}_*\zeta=\overleftrightarrow{c}_*\zeta'\in\Corr(Z;\const{X},\const{Y})$. It remains to apply Lem.
\ref{l:pushclass} to $\overleftrightarrow{c}:C\to Z$.
\end{proof}

\begin{exam}\label{ex:graph}
Let $\phi:X\to Y$ be a morphism over $S$ and $\Gamma(\phi)$ be its graph in $X\times_SY$, then $\Gamma(\phi)$ is obviously a left graph-like correspondence between $X$ and $Y$. Let $[\Gamma(\phi)]\in \cohog{0}{\Gamma(\phi),\DD_{\Gamma(\phi)/X}}=\cohog{0}{\Gamma(\phi)}$ be the class of the constant function 1, or the fundamental class of $\Gamma(\phi)$ relative to $X$. Then the homomorphism $[\Gamma(\phi)]_\#:g_!\const{Y}\to f_*\const{X}$ is
\begin{equation*}
[\Gamma(\phi)]_\#:g_!\const{Y}\to g_*\const{Y}\xrightarrow{\phi^*}f_*\const{X}.
\end{equation*}
\end{exam}

\subsubsection{Integration along a correspondence} For a correspondence $C$ between $X$ and $Y$ over $S$ satisfying (G-2) with respect to some $U\subset S$, we have the fundamental class $[C_U]\in \hBM{2d}{C}(-d)$, defined as the sum of the fundamental classes of the closures of $d$-dimensional irreducible components of $C_U$. Using the fundamental class of $X$, we can identify $\const{X}$ with $\DD_X[-2d](-d)$, and get a quasi-isomorphism
\begin{equation*}
\DD_{\overleftarrow{c}}\cong\DD_C[-2d](-d).
\end{equation*}
Therefore $[C_U]$ can be viewed as a class in $\cohog{0}{C,\DD_{\overleftarrow{c}}}=\Corr(C;\const{X},\const{Y})$. We claim that the induced map $[C_U]_\#:g_!\const{Y}\to f_*\const{X}$ is independent of $U$. In fact, if $C$ also satisfies the condition (G-2) with respect to another $V\subset S$, then it again satisfies (G-2) with respect to $U\cap V$. Since $[C_U]$ and $[C_V]$ both restrict to $[C_{U\cap V}]$ in $\Corr(C_{U\cap V};\Ql,\Ql)$, Lem. \ref{l:openpart} implies that $[C_U]_\#=[C_V]_\#$. Therefore it is unambiguous to write
\begin{equation*}
[C]_\#:g_!\const{Y}\to f_*\const{X},
\end{equation*}
which is the sheaf-theoretic analogue of {\em integration along the correspondence $C$}. 

Now we study the composition of such integrations. We use the notation in the diagram (\ref{d:comp}). Let $X,Y$ be smooth, equidimensional and $Y$ be proper over $S$.

\begin{prop}\label{p:comp}
Assume $C_2$ is left graph-like and $C_1, C=C_1*C_2$ satisfy condition (G-2) with respect to some $U\subset S$, then
\begin{equation*}
[C_1]_\#\circ[C_2]_\#=[C]_\#:h_!\const{Z}\to f_*\const{X}.
\end{equation*}
Similarly, if we assume $C_1$ is right graph-like and $C_2, C=C_1*C_2$ satisfy condition (G-2) with respect to some $U\subset S$, the same conclude holds.
\end{prop}
\begin{proof}
We prove the first statement. The proof follows from a sequence of d\'{e}vissages. By Lem. \ref{l:conv}, it suffices to prove
\begin{equation*}
[C_1]\circ_Y[C_2]=[C]\in\Corr(C;\const{X},\const{Z}).
\end{equation*}
By property (G-2) and Lem. \ref{l:openpart}, it suffices to prove
\begin{equation}\label{eq:twosides}
[C_{1,U}]\circ_{Y_U}[C_{2,U}]=[C_U]\in\cohog{0}{C_U,\DD_{\overleftarrow{c}}}.
\end{equation}
Therefore we have reduced to the case where $\overleftarrow{c_2}$, and hence $\overleftarrow{d}$ are \'{e}tale. In this case, we can identify $\DD_{\overleftarrow{c_2}}$ with $\const{C_2}$. Under this identification, the convolution product becomes
\begin{equation*}
\cohog{0}{C_1,\DD_{\overleftarrow{c_1}}}\otimes\cohog{0}{C_2}\to\cohog{0}{C,\DD_{\overleftarrow{c}}}.
\end{equation*}
and $[C_2]$ becomes the class of constant function 1 in $\cohog{0}{C_2}$. Therefore, convolution with $[C_2]$ becomes the pull-back along the \'{e}tale morphism $\overleftarrow{d}$:
\begin{equation*}
\overleftarrow{d}^*:\cohog{0}{C_1,\DD_{\overleftarrow{c_1}}}\to\cohog{0}{C,\DD_{\overleftarrow{c}}}.
\end{equation*}
It is obvious that $\overleftarrow{d}^*[C_1]=[C]$. Therefore (\ref{eq:twosides}) is proved.
\end{proof}

\begin{remark}
The proposition fails if we only assume $\overleftarrow{c_2}:C_2\to Y$ to be quasi-finite. For example, take $X=Y=Z=\PP^1$ and $S=\textup{pt}$. Let $C_2$ be the union of the diagonal and the graph of $z\mapsto z^{-1}$ in $\PP^1\times\PP^1$. Let $C_1$ be the graph of the constant map to $1\in\PP^1$. The the reduced structure of the composition $C$ is the same as $C_1$, the constant graph. However, the action of $[C]_\#$ on $\cohog{0}{X}$ is the identity while the action of $[C_1]_\#\circ[C_2]_\#$ on $\cohog{0}{X}$ is twice the identity.
\end{remark}

% convolution algebra 

\subsection{The convolution algebra}\label{ss:convalg}
Assume $X$ is smooth of equidimension $d$ and $f:X\to S$ is proper. Let $C$ be a self correspondence of $X$ over $S$ satisfying (G-2). Assume we have a morphism $\mu:C*C\to C$ as correspondences which is associative, i.e., the following diagram is commutative:
\begin{equation*}
\xymatrix{& C*C*C\ar[dl]_{\mu*\id}\ar[dr]^{\id*\mu} &\\
C*C\ar[dr]^{\mu} & & C*C\ar[dl]_{\mu}\\
& C &}
\end{equation*}
Then the convolution gives a multiplication on $\Corr(C;\const{X},\const{X})$:
\begin{eqnarray*}
\circ:\Corr(C;\const{X},\const{X})\otimes\Corr(C;\const{X},\const{X})&\xrightarrow{\circ_X}&\Corr(C*C;\const{X},\const{X})\\
&\xrightarrow{\mu_*}&\Corr(C;\const{X},\const{X}).
\end{eqnarray*}
This multiplication is associative by Lem. \ref{l:ass} and the assumption that $\mu$ is associative. Therefore
$\Corr(C;\const{X},\const{X})$ acquires a (non-unital) algebra structure. Restricting to $C_U$, the vector space $\Corr(C_U;\const{X_U},\const{X_U})$ also acquires a (non-unital) algebra structure. 

\begin{remark}\label{rm:inv}
We have a map
\begin{equation*}
\End_S(X)\to\Corr(X\times_S X;\const{X},\const{X})
\end{equation*}
which sends a morphism $\phi:X\to X$ to the fundamental class of its graph. By Example \ref{ex:graph}, this is an {\em anti-}homomorphism of monoids. Here the monoid structures are given by the composition of morphisms on the LHS and the convolution $\circ$ on the RHS.
\end{remark}

\begin{prop}\label{p:alghom}
\begin{enumerate}
\item []
\item The map
\begin{equation}\label{eq:corract}
(-)_\#:\Corr(C;\const{X},\const{X})\to\End_S(f_*\const{X})
\end{equation}
is an algebra homomorphism.
\item The map (\ref{eq:corract}) factors through the restriction map
\begin{equation*}
j^*:\Corr(C;\const{X},\const{X})\to\Corr(C_U;\const{X_U},\const{X_U})
\end{equation*}
so that we also have an algebra homomorphism
\begin{equation}\label{eq:copenact}
(-)_{U,\#}:\cohog{0}{C_U,\DD_{\overleftarrow{c}}}\to\End_S(f_*\const{X}).
\end{equation}
\end{enumerate}
\end{prop}
\begin{proof}
(1) follows from Lem. \ref{l:conv}; 
(2) The factorization follows from Lem. \ref{l:openpart}. Using the identification $\const{X}\cong\DD_X[-2d](-d)$, we can identify $j^*$ with the restriction map
\begin{equation*}
\hBM{2d}{C}(-d)\to\hBM{2d}{C_U}(-d),
\end{equation*}
which is obviously surjective. Therefore the map (\ref{eq:copenact}) is also an algebra homomorphism because the map (\ref{eq:corract}) is.
\end{proof}

% references

\end{document}